\documentclass[11pt,oneside]{amsart}
\usepackage{amsmath}
\usepackage{amssymb}
\usepackage{amscd}
\usepackage{amsthm}
\usepackage{latexsym}
\usepackage[pdftex]{graphicx}
\usepackage{pinlabel}
\usepackage{enumerate}
\usepackage{epstopdf}
\usepackage{enumitem}

\usepackage{amsbsy}

\numberwithin{equation}{section}
\theoremstyle{plain}

\newtheorem{thm}{Theorem}[section]

\newtheorem{thm-defn}[thm]{Theorem/Definition}
\newtheorem{lemma}[thm]{Lemma}
\newtheorem{prop}[thm]{Proposition}
\newtheorem{cor}[thm]{Corollary}
\newtheorem{conj}[thm]{Conjecture}

\theoremstyle{definition}
\newtheorem{defn}[thm]{Definition}

\newtheorem{example}[thm]{Example}

\theoremstyle{remark}
\newtheorem{rmk}[thm]{Remark}

\newcommand{\A}{{\mathcal A}}

\newcommand{\C}{\mathbb C}

\newcommand{\E}{\mathcal E}

\newcommand{\J}{{\mathcal J}}

\renewcommand{\L}{{\mathcal L}}

\newcommand{\Q}{\mathbb Q}

\newcommand{\R}{\mathbb R}

\renewcommand{\S}{\mathcal S}

\newcommand{\X}{{\mathcal X}}
\newcommand{\Z}{\mathbb Z}

\newcommand{\oo}{\infty}

\begin{document}

\begin{abstract}
The rational ho\-mol\-ogy balls $B_n$ ap\-peared in Fin\-tu\-shel and Stern's ration\-al blow-down construction \cite{FSrbd}. Later, Symington \cite{Sym1}, defined this operation in the symplectic category. In \cite{Kh2}, the author defined the inverse procedure, the symplectic rational blow-up. In this paper, we study the obstructions to symplectically rationally blowing up a symplectic $4$-manifold, i.e. the obstructions to symplectically embedding the rational homology balls $B_n$ into a symplectic $4$-manifold. We prove a theorem and give additional examples which suggest that in order to symplectically embed the rational homology balls $B_n$, for high $n$, a symplectic 4-manifold must at least have a high enough $c_1^2$ as well.
\end{abstract}

\title[Bounds on Embeddings of $\mathbb{Q}$HB in Symplectic 4-manifolds]{Bounds on Embeddings of Rational Homology Balls in Symplectic 4-manifolds}
\author{Tatyana Khodorovskiy}

\maketitle

\section{Introduction}
In 1997, Fintushel and Stern \cite{FSrbd} defined the rational blow-down operation for smooth $4$-manifolds, a generalization of the standard blow-down operation. For smooth $4$-manifolds, the standard blow-down is performed by removing a neighborhood of a sphere with self-intersection $(-1)$ and replacing it with a standard $4$-ball $B^4$. The rational blow-down involves replacing a negative definite plumbing $4$-manifold with a rational homology ball. In order to define it, we first begin with a description of the negative definite plumbing $4$-manifold $C_n$, $n \geq 2$, as seen in Figure~\ref{f:cn}, where each dot represents a sphere, $S_i$, in the plumbing configuration. The integers above the dots are the self-intersection numbers of the plumbed spheres: $[S_1]^2 = -(n+2)$ and $[S_i]^2 = -2$ for $2 \leq i \leq n-1$.

\begin{figure}[ht!]
\labellist
\small\hair 2pt
\pinlabel $-(n+2)$ at 30 4.7
\pinlabel $-2$ at 60 4.7
\pinlabel $-2$ at 90 4.7
\pinlabel $-2$ at 173 4.7
\pinlabel $-2$ at 203 4.7
\pinlabel $S_1$ at 32 2.5 
\pinlabel $S_2$ at 62 2.5 
\pinlabel $S_3$ at 92 2.5 
\pinlabel $S_{n-2}$ at 175 2.5 
\pinlabel $S_{n-1}$ at 205 2.5 
\endlabellist
\centering
\includegraphics[height=30mm, width=120mm]{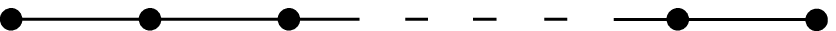}
\caption{{\bf Plumbing diagram of $C_n$, $n\geq2$}}
\label{f:cn}
\end{figure}

The boundary of $C_n$ is the lens space $L(n^2,n-1)$, thus $\pi_1(\partial C_n) \cong H_1(\partial C_n ; \Z) \cong \Z/n^2\Z$. (Note, when we write the lens space $L(p,q)$, we mean it is the $3$-manifold obtained by performing $-\frac{p}{q}$ surgery on the unknot.) This follows from the fact that $[-n-2, -2, \ldots -2]$, with $(n-2)$ many $(-2)$'s is the continued fraction expansion of $\frac{n^2}{1-n}$. (Note, we will often abuse notation and write $C_n$ both for the actual plumbing $4$-manifold and the plumbing configuration of spheres in that $4$-manifold.)

\begin{figure}[ht!]
\labellist
\small\hair 2pt
\pinlabel $n-1$ at -25 240
\pinlabel $n$ at 235 160
\endlabellist
\centering
\includegraphics[scale=0.25]{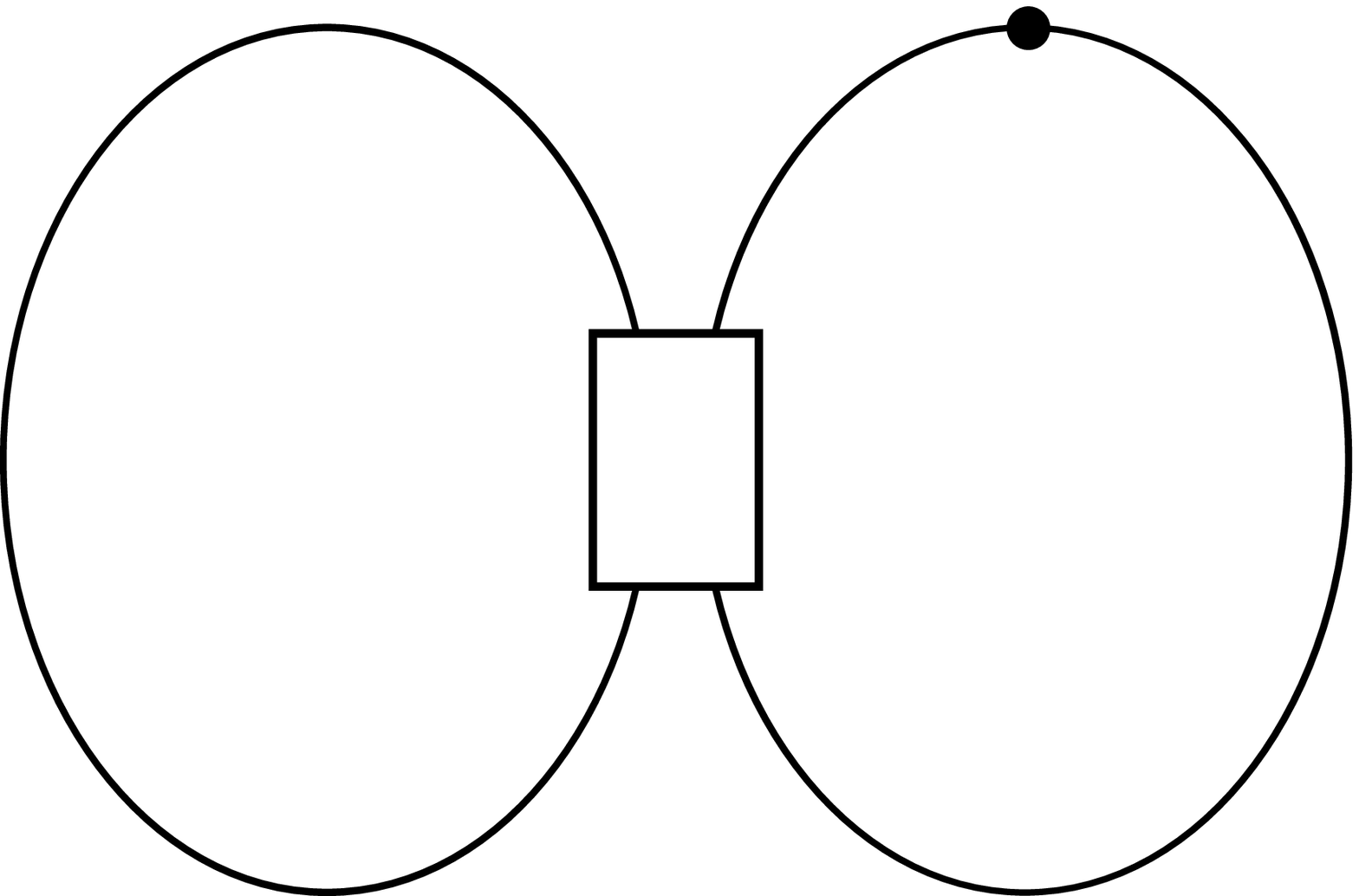}
\caption{{\bf Kirby diagram of $B_n$}}
\label{f:bn}
\end{figure}

Let $B_n$ be the $4$-manifold as defined by the Kirby diagram in Figure~\ref{f:bn} (for a more extensive description of $B_n$, see section~\ref{sec:bn}). The manifold $B_n$ is a rational homology ball, i.e. $H_*(B_n;\Q) \cong H_*(B^4;\Q)$. The boundary of $B_n$ is also the lens space $L(n^2,n-1)$ \cite{CaHa}. Moreover, any self-diffeomorphism of $\partial B_n$ extends to $B_n$ \cite{FSrbd}. Now, we can define the rational blow-down of a $4$-manifold $X$:

\begin{defn}
(\cite{FSrbd}, also see \cite{GS}) Let $X$ be a smooth $4$-manifold. Assume that $C_n$ embeds in $X$, so that $X = C_n \cup_{L(n^2,n-1)} X_0$. The 4-manifold $X_{(n)} = B_n \cup_{L(n^2,n-1)} X_0$ is by definition the \textit{rational blow-down} of $X$ along the given copy of $C_n$.
\end{defn}

Fintushel and Stern \cite{FSrbd} also showed how to compute Seiberg-Witten and Donaldson invariants of $X_{(n)}$ from the respective invariants of $X$. In 1998, Symington \cite{Sym1} proved that the rational blow-down operation can be performed in the symplectic category. More precisely, she showed that if in a symplectic 4-manifold $(M,\omega)$ there is a symplectic embedding of a configuration $C_n$ of symplectic spheres, then there exists a symplectic model for $B_n$ such that the \textit{rational blow-down} of $(M, \omega)$, along $C_n$ is also a symplectic $4$-manifold. In \cite{Kh2}, the author defined the \textit{symplectic rational blow-up} operation, where the symplectic structure of $B_n$ is presented as an entirely standard symplectic neighborhood of a certain Lagrangian 2-cell complex, enabling one to replace the $B_n$ with $C_n$ and obtain a new symplectic 4-manifold. 

The main goal of this paper is is to investigate the following question: \textbf{what are the obstructions to symplectically embedding the rational homology balls $B_n$ into a symplectic $4$-manifold?} Note, in \cite{Kh1}, the author showed that in the smooth category there is little obstruction to embedding a rational homology ball $B_n$:

\begin{thm}
\label{thm:smoothbn}
\cite{Kh1} Let $V_{-4}$ be a neighborhood of a sphere with self-intersec\-tion number $(-4)$. For all $n \geq 3$ odd, there exists an embedding of the rational homology balls $B_n \hookrightarrow V_{-4}$. For all $n \geq 2$ even, there exists an embedding of the rational homology balls $B_n \hookrightarrow B_2 \# \overline{\C P^2}$.
\end{thm}

\noindent Theorem~\ref{thm:smoothbn} above implies that if a smooth $4$-manifold $X$ contains a sphere with self-intersection $(-4)$, then one can smoothly embed the rational homology balls $B_n$ into $X$ for all odd $n \geq 3$. One of the implications of this is that for a given smooth $4$-manifold $X$, there does not exist an $N$, such that for all $n \geq N$ one cannot find a smooth embedding $B_n \hookrightarrow X$. In the setting of this sort in algebraic geometry, for rational homology ball smoothings of certain surface singularities, such a bound on $n$ does exist, in terms of $(c_1^2,\chi_h)$ invariants of an algebraic surface \cite{KoSB,Wa}. Therefore, for the case of symplectic embeddings of the rational homology balls $B_n$, if we model our symplectic manifold such that it resembles a surface of general type, we can make the following conjecture:

\begin{conj}
\label{conj:main}
Let $(X,\omega)$ be a symplectic $4$-manifold, such that:
\begin{itemize}
\item
$b_2^+(X) > 1$ and
\item
$\left[c_1(X,\omega)\right] = -\left[\omega \right]$ as cohomology classes,
\end{itemize}
then there exists an $N$, such that for all $n \geq N$ there does not exist a symplectic embedding $B_n \hookrightarrow (X,\omega)$.
\end{conj}

\noindent The condition $\left[c_1(X,\omega)\right] = -\left[\omega \right]$, implies that $(X,\omega)$ does not contain any spheres of self-intersection $(-1)$ or $(-2)$ and $c_1^2(X,\omega) \geq 1$, resembling a surface of general type with an ample canonical divisor. 

We prove a result (Theorem~\ref{thm:main}) that is a first step in proving the above conjecture. We observe that if we impose the condition $n \geq c_1^2(X,\omega) + 2$ on $(X,\omega)$, then if we symplectically rationally blow up a $B_n \hookrightarrow (X,\omega)$, we would obtain a symplectic manifold $(X',\omega')$ for which $c_1^2(X',\omega') \leq -1$. As a consequence of a theorem of Taubes \cite{Ta1,Ta2,Ta4}, we would then obtain, for a generic $\omega$-compatible almost-complex structure $J_{\epsilon}$, a $J_{\epsilon}$-holomorphic embedded sphere $\Sigma_{-1}^{\epsilon}$ with self-intersection $(-1)$. The consequences of the existence of such a sphere in the symplectic rational blow-up $(X',\omega')$ leads to various contradictions of adjunction formulas and results on Seiberg-Witten invariants.

We show that if $(X,\omega)$ is such that $n \geq c_1^2(X,\omega) + 2$ (in addition to the two conditions on $(X,\omega)$ in Conjecture~\ref{conj:main}), then a symplectic embedding $B_n \hookrightarrow (X,\omega)$ will fall into two types: $\A$ and $\E_k$, $2 \leq k \leq n-1$, (see Definitions~\ref{d:bntype},~\ref{d:bntypeA}, ~\ref{d:bntypeEk}). The types $\A$ and $\E_k$ are determined by the intersection patterns of a sphere $\Sigma_{-1}$, with self-intersection $(-1)$ (obtained as consequence of the sphere $\Sigma_{-1}^{\epsilon}$), with the spheres of $C_n \subset (X',\omega')$. We then prove the following theorem:

\begin{thm}
\label{thm:main}
If $B_n \hookrightarrow (X,\omega)$ is a symplectic embedding, where  $(X,\omega)$ is a symplectic $4$-manifold, such that:
\begin{itemize}
\item
$b_2^+(X) > 1$,
\item
$\left[c_1(X,\omega)\right] = -\left[\omega \right]$ as cohomology classes,
\item
$n \geq c_1^2(X,\omega) + 2$ and
\item
$\mathcal{B}as_X = \left\{\pm c_1(X,\omega)\right\}$, ($\mathcal{B}as_X$ denotes the set of Seiberg-Witten basic classes of X,)
\end{itemize}
then it cannot be of type $\A$ or of type $\E_k$, $k \geq c_1^2(X,\omega) + 2$.
\end{thm}

\noindent Note, in Theorem~\ref{thm:main} above, the condition $\mathcal{B}as_X = \left\{\pm c_1(X,\omega)\right\}$ on $(X,\omega)$ is also true for surfaces of general type. 

We also describe a family of symplectic manifolds, $\X$, constructed from the elliptic surfaces $E(m)$, which contain an embedded $B_n$ of type $\E_2$ (not covered by Theorem~\ref{thm:main}), in such a way that 
\begin{equation*}
n < 3 + \frac{4}{3}c_1^2(X,\omega) \, ,  
\end{equation*}

\noindent for all $(X,\omega) \subset \X$. Thus, also providing evidence for Conjecture~\ref{conj:main}, that every symplectic manifold has a bound on $n$, above which one can no longer embed a rational homology ball $B_n$. Both Theorem~\ref{thm:main} and this family of examples suggest that in order for there to exist a symplectic embedding $B_n \hookrightarrow (X,\omega)$ for high $n$, the manifold $(X,\omega)$ needs to at least have a high enough $c_1^2(X,\omega)$.

It is worthwhile to note, that obstructions to symplectically embedding the rational homology balls $B_n$, was the subject of some recent research \cite{LeMa}. Their obstructions arose from nonvanishing symplectic cohomology, however, their results do not depend on $n$.

This paper is organized as follows. In section~\ref{sec:background}, we give some brief reviews: the structure of the rational homology balls $B_n$ and the symplectic rational blow-up construction appearing in \cite{Kh2}; Seiberg-Witten invariants and basic classes; and toric and almost-toric fibrations of symplectic 4-manifolds, which is used the proof of the main theorem.

In section~\ref{sec:sympemb}, after separating the symplectic embeddings of $B_n \hookrightarrow (X,\omega)$ into types $\A$ and $\E_k$, we prove the main theorem (Theorem~\ref{thm:main}). We prove Theorem~\ref{thm:main} in four steps, by assuming that there exists a symplectic embedding $B_n \hookrightarrow (X,\omega)$ and obtaining a contradiction. In \textit{Step 1}, section~\ref{sec:step1}, we show that symplectic embeddings of $B_n$ will indeed be of type $\A$ or $\E_k$. In \textit{Step 2}, section~\ref{sec:step2}, we construct a cycle $\gamma$ and compute $c_1(X,\omega) \cdot \gamma$. In \textit{Step 3}, section~\ref{sec:step3}, we show that if $c_1(X,\omega) \cdot \gamma > 0 $ then $\omega \cdot \gamma > 0$, contradicting the $\left[c_1(X,\omega)\right] = -\left[\omega \right]$ assumption. In \textit{Step 4}, section~\ref{sec:step4}, we show that if $c_1(X,\omega) \cdot \gamma \leq 0$, then the condition $\mathcal{B}as_X = \left\{\pm c_1(X,\omega)\right\}$ or the adjunction formula will be violated. Additionally, in section~\ref{sec:e2}, we provide explicit examples of symplectic embeddings of $B_n \hookrightarrow (X,\omega)$ of type $\E_2$, which adhere to Conjecture~\ref{conj:main}.

\section{Background}
\label{sec:background}
\subsection{Description of the rational homology balls $B_n$}
\label{sec:bn}
There are several ways to give a description of the rational homology balls $B_n$. One of them is a Kirby calculus diagram seen in Figure~\ref{f:bn}. This represents the following handle decomposition: Start with a 0-handle, a standard 4-disk $D^4$, attach to it a 1-handle $D^1 \times D^3$. Call the resultant space $X_1$, it is diffeomorphic to $S^1 \times D^3$ and has boundary $\partial X_1 = S^1 \times S^2$. Finally, we attach a 2-handle $D^2 \times D^2$. The boundary of the core disk of the 2-handle gets attached to the closed curve, $K$, in $\partial X_1$ which wraps $n$ times around the $S^1 \times \ast$ in $S^1 \times S^2$. We can also represent $B_n$ by a slightly different Kirby diagram, which is more cumbersome to manipulate but is more visually informative, as seen in Figure~\ref{f:bnsph}, where the 1-handle is represented by a pair of balls.

\begin{figure}[ht!]
\labellist
\small\hair 2pt
\pinlabel $n-1$ at 475 18
\pinlabel $\}n$ at 190 65
\endlabellist
\centering
\includegraphics[scale=0.5]{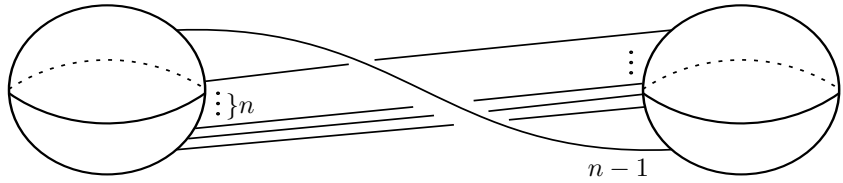}
\caption{{\bf Another Kirby diagram of $B_n$}}
\label{f:bnsph}
\end{figure}

\begin{figure}[ht]
\begin{minipage}[b]{0.45\linewidth}
\centering
\includegraphics[scale=0.30]{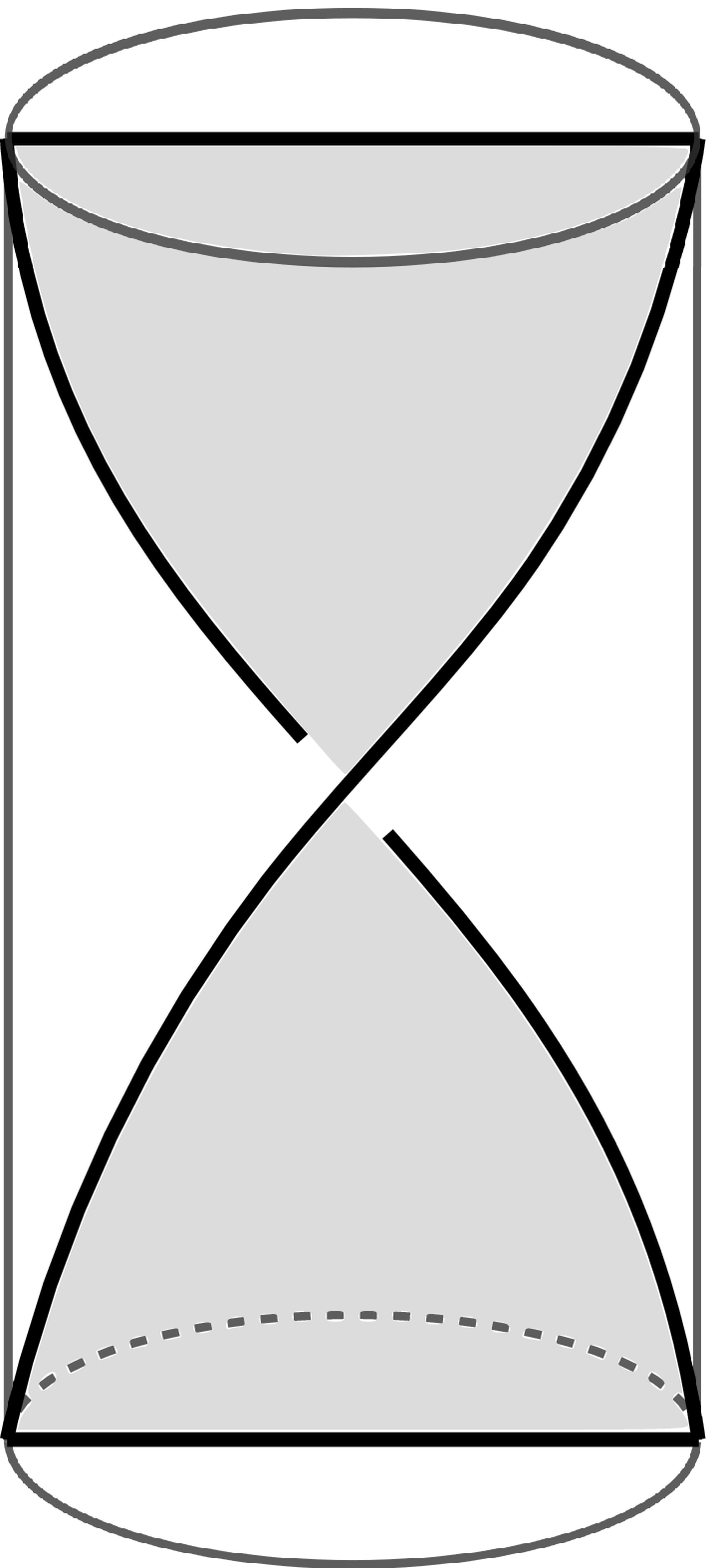}
\caption{$L'_2$}
\label{f:L2}
\end{minipage}
\hspace{0.5cm}
\begin{minipage}[b]{0.45\linewidth}
\centering
\includegraphics[scale=0.30]{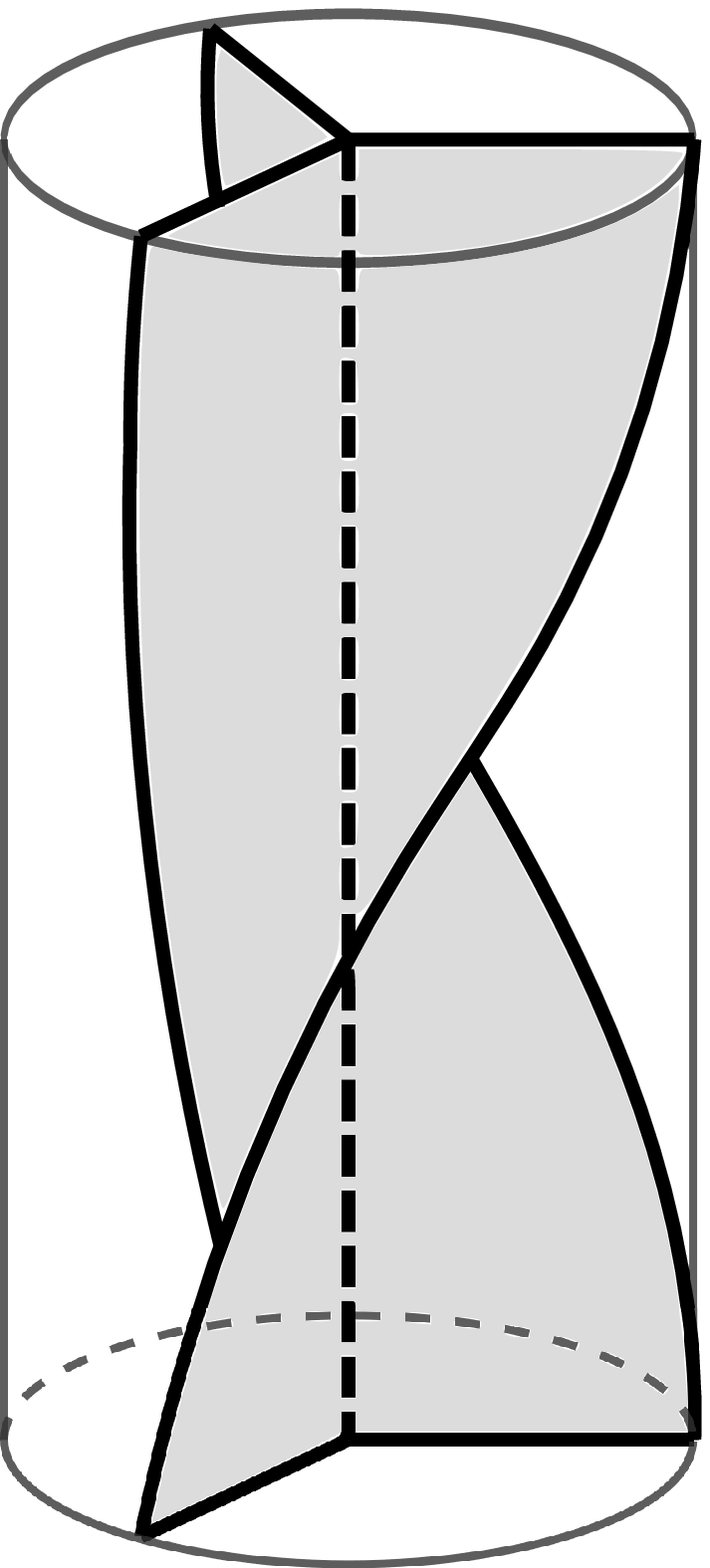}
\caption{$L'_3$}
\label{f:L3}
\end{minipage}
\end{figure}

The rational homology ball $B_2$ can also be described as an unoriented disk bundle over $\mathbb{R}P^2$. Since $\mathbb{R}P^2$ is the union of a Mobius band $M$ and a disk $D$, we can visualize $\mathbb{R}P^2$ sitting inside $B_2$, with the Mobius band and its boundary $(M, \partial M)$ embedded in $(X_1 \cong S^1 \times D^3, \partial X_1 \cong S^1 \times S^2)$ (Figure~\ref{f:L2}, with the ends of the cylinder identified), and the disk $D$ as the core disk of the attaching 2-handle. We will construct something similar for $n\geq 3$. Instead of the Mobius band sitting inside $X_1$, as for $n=2$, we have a ``$n$-Mobius band" (a Moore space), $L'_n$, sitting inside $X_1$. The case of $n=3$ is illustrated in Figure~\ref{f:L3}, again with the ends of the cylinder identified. In other words, $L'_n$ is a singular surface, homotopic to a circle, in $X_1 \cong S^1 \times D^3$, whose boundary is the closed curve $K$ in $\partial X_1 \cong S^1 \times S^2$, and it includes the circle, $S = S^1 \times 0$ in $S^1 \times D^3$. Let $L_n = L'_n \cup_K D$, where $D$ is the core disk of the attached 2-handle (along $K$). We will call $L_n$ the core of the rational homology ball $B_n$; observe, that $L_2 \cong \mathbb{R}P^2$.

These cores $L_n$ were used as geometrical motivation in the construction of a symplectic structure on the rational homology balls $B_n$, in the definition of the symplectic rational blow-up operation \cite{Kh2}. For $n=2$, if we have an embedded $\mathbb{R}P^2$ in $(X,\omega)$, such that $\omega |_{\mathbb{R}P^2} = 0$, (i.e. a Lagrangian $\mathbb{R}P^2$) then the $\mathbb{R}P^2$ will have a totally standard neighborhood, which will be symplectomorphic to the rational homology ball $B_2$. In this vein, for $n \geq 3$, we can define $\L_n$ (labeled $\L_{n,1}$ in \cite{Kh2}) as a cell complex consisting of an embedded $S^1$ and a 2-cell $D^2$, whose boundary ``wraps" $n$ times (winding number) around the embedded $S^1$ (the interior of the 2-cell $D^2$ is an embedding). Furthermore, the cell complex $\L_n$ is embedded in such a way that the 2-cell $D^2$ is Lagrangian. It is shown in \cite{Kh2}, by mirroring the Weinstein Lagrangian embedding theorem, that a symplectic neighborhood of such an $\L_n$ is entirely standard, and is a symplectic model for $B_n$. Therefore, given the existence of such an $\L_n$, we can replace the $B_n$ with $C_n$ and obtain a new symplectic 4-manifold $(X',\omega')$, the symplectic rational blow-up of $(X,\omega)$.

\subsection{Review of Seiberg-Witten invariants and basic classes}
\label{sec:sw}
Here we give a brief overview of Seiberg-Witten invariants and basic classes, and state some relevant results. For a full description of Seiberg-Witten invariants see \cite{Mo}, and for a short overview see \cite{GS}, section 2.4 (which this summary is based on). The Seiberg-Witten invariant is a powerful invariant of smooth manifolds. More precisely, these are invariants of a smooth $4$-manifold together with a $spin^c$ structure. 

We let $X$ be a smooth, closed, oriented 4-manifold, with $b_2^+(X)>1$ odd. Given a $spin^c$ structure $\mathfrak{s}$, we can associate to it a \textit{determinant line bundle} $L$. If $H^2(X;\Z)$ has no $2$-torsion, then the set of $spin^c$ structures of $X$, $\S^c(X)$, is in $1$-$1$ correspondence (via $c_1(L)$) with the set of characteristic elements of $X$, $\mathcal{C}_X$:

\begin{defn}
The set of \textit{characteristic elements} of $X$ (as above) is:
\begin{equation}
\mathcal{C}_X = \left\{K \in H^2(X;\Z) | K \equiv w_2(X) (\text{mod} 2)\right\} \, .
\end{equation}
\end{defn}

We will assume for simplicity of the exposition that $H^2(X;\Z)$ has no $2$-torsion. Let $\mathcal{M}_X^{\delta,g}(K)$ be the moduli space of solutions to certain perturbed monopole equations, where $K \in \mathcal{C}_X$, $g$ is a given metric on $X$ and $\delta \in \Omega^+(X)$ is a perturbation. The moduli space $\mathcal{M}_X^{\delta,g}(K)$ is itself a closed and orientable manifold (for a generic metric $g$) of dimension $\frac{1}{4}(K^2 - (3\sigma(X) + 2\chi(X)))$. In addition, $\mathcal{M}_X^{\delta,g}(K)$ is a subspace of an infinite-dimensional manifold $\mathcal{B}^{\ast}_K$, which is homotopy equivalent to $\C P^{\oo}$, in particular, implying that $H^{\ast}(\mathcal{B}^{\ast}_K ; \Z) \cong \Z[\mu]$ and $[\mathcal{M}_X^{\delta,g}(K)] \in H_{2m}(\mathcal{B}^{\ast}_K ; \Z)$ is a homology class. 

\begin{defn}
For $X$ as above, the \textit{Seiberg-Witten invariant} is $SW_X : \mathcal{C}_X \rightarrow \Z$ is defined by $SW_X(K) = \left\langle \mu^m,[\mathcal{M}_X^{\delta,g}(K)]\right\rangle$, where dim $\mathcal{M}_X^{\delta,g}(K) = 2m$ and if dim $\mathcal{M}_X^{\delta,g}(K) < 0$ then $SW_X(K) = 0$. (If dim $\mathcal{M}_X^{\delta,g}(K)$ is odd then $b^{+}_2(X)$ is even, and we are assuming $b^{+}_2(X)$ is odd.)
\end{defn}

 The \textit{Seiberg-Witten invariant} is $SW_X$ is indeed a diffeomorphism invariant: it does not depend on the choices made in its construction.

\begin{defn}
A cohomology class $K \in \mathcal{C}_X \subset H^2(X;\Z)$ is a \textit{Seiberg-Witten basic class} if $SW_X(K) \neq 0$, and the set of basic classes denoted by $\mathcal{B}as_X$.
\end{defn}

\begin{defn}
A simply connected $4$-manifold is said to be of \textit{simple type} if for each $K \in \mathcal{B}as_X$ we have $K^2 = c_1^2(X) = 3\sigma(X) + 2\chi(X)$ (implying that dim $\mathcal{M}_X^{\delta,g}(K) = 0$).
\end{defn}

Now we will state some useful results of Seiberg-Witten invariants:

The Seiberg-Witten invariants behave very well under blow-ups (\cite{FS2} for general case):

\begin{thm}{\textbf{The blow-up formula}}
\cite{GS}. Let $X$ be a simply connected $4$-manifold of simple type with $\mathcal{B}as_X = \left\{K_i | i = 1, \ldots , s\right\}$. If $X' = X \# \overline{\C P^2}$ is the blow-up of $X$ and $E \in H^2(X';\Z)$ denotes the Poincar\`{e} dual of the homology class $e \in H_2(X',\Z)$ of the exceptional sphere, then the set of basic classes of $X'$ equals $\left\{K_i \pm E | i = 1, \ldots, s\right\}$.
\end{thm}

For Seiberg-Witten behavior under rational blow-downs, we have the following results, \cite{FSrbd}, also see \cite{GS}:

\begin{prop}
Let the sphere configuration $C_n \subset X$, and $X_{(n)} = X^{\circ} \cup B_n$ (where $X^{\circ} = X - C_n$) be the rational blow-down of $X$ along $C_n$. Then for every characteristic element $\overline{K} \in \mathcal{C}_{X_{(n)}}$ there is an element $K \in \mathcal{C}_X$ such that $\overline{K}_{|X^{\circ}} = K_{|X^{\circ}}$ and $K^2 - \overline{K}^2 = -(n-1)$. The class $K$ is called a \textit{lift} of $\overline{K}$.
\end{prop}

\begin{thm}
\label{thm:FSrbdsw}
Suppose that $X$ and $X_{(n)}$ (as above) are simply connected $4$-mani\-folds. Choose $\overline{K} \in \mathcal{C}_{X_{(n)}}$, and fix a \text{lift} $K \in \mathcal{C}_X$ for it. If $K^2 \geq 3\sigma(X) + 2\chi(X)$, then $SW_{X_{(n)}}(\overline{K}) = SW_X(K)$. Consequently, the Seiberg-Witten invariants of $X$, $SW_X$, determine the Seiberg-Witten invariants of the rational blow-down of $X$, $SW_{X_{(n)}}$. 
\end{thm}

\begin{rmk}
The theorem above expresses the SW basic classes of $X_{(n)}$ in terms of the SW basic classes of $X$. Consequently, it tells us which SW basic classes $X$ ``pass down" to $X_{(n)}$. It does not, however, provide us a way to reconstruct the SW basic classes of $X$ from those of $X_{(n)}$. In fact, the only basic classes that can ``pass down" from $X$ to $X_{(n)}$, are those which when restricted to $\partial X^{\circ} \cong L(n^2,n-1)$, correspond to an element of order $n$ in $H^2(L(n^2,n-1),\Z) \cong \Z / n^2 \Z$. 
\end{rmk}

For complex surfaces $S$, and a smooth, nonsingular, connected, complex curve $C \subset S$, the standard adjunction formula says that:
\begin{equation*}
2g(C)-2 = [C]^2 - \left\langle c_1(S),C \right\rangle \, , 
\end{equation*}

\noindent where $g(C)$ is the genus of $C$ \cite{GS}. The Seiberg-Witten invariants give us the following adjunction formula result for smooth manifolds $X$: 

\begin{thm}{\textbf{Generalized adjunction formula}}
\label{thm:genadjfor}
\cite{KM,OzSz}, also see \cite{GS}. Assume that $\Sigma \subset X$ is an embedded, oriented, connected surface of genus $g(\Sigma)$ with self-intersection $[\Sigma]^2 \geq 0$ (and $[\Sigma] \neq 0$). Then for every Seiberg-Witten basic class $K \in \mathcal{B}as_X$ we have $2g(\Sigma) - 2 \geq [\Sigma]^2 + |K(\Sigma)|$. If $X$ is of simple type and $g(\Sigma) > 0$, the same inequality holds for $\Sigma \subset X$ with arbitrary square $[\Sigma]^2$.
\end{thm}

There is also further generalization of this result for immersed spheres. We state here a simplified version, where dim $\mathcal{M}_X^{\delta,g}(K)=0$:

\begin{thm}
\label{thm:genadjimm}
\cite{FS2}. {\textbf{Generalized adjunction formula for }} {\textbf{ immersed spheres.}} Suppose that $X$ is an arbitrary smooth $4$-manifold with $b_2^+(X) > 1$ and that $K \in \mathcal{C}_X$ with $SW_X(K) \neq 0$ and dim $\mathcal{M}_X(K)=0$. If $x \neq 0 \in H_2(X;\Z)$ is represented by an immersed sphere with $p$ positive double points, then either
\begin{equation*}
2p-2 \geq x^2 + |x \cdot L| 
\end{equation*}

\noindent or
\begin{equation*}
 SW_X(K)=\begin{cases}
    SW_X(K+2x), & \text{if} \,\, x\cdot K \geq 0\\
    SW_X(K-2x), & \text{if} \,\, x\cdot K \leq 0.
  \end{cases}     
\end{equation*}
\end{thm}

The Seiberg-Witten invariants also have interesting behavior if the $4$-manifold $X$ is equipped with a symplectic form $\omega$. For example, if a $4$-manifold has a symplectic structure then it must be of \textit{simple type}. Additionally, we have the following important results of Taubes:

\begin{thm}
\cite{Ta3} If $(X,\omega)$ is a simply connected symplectic mani\-fold with $b_2^+(X) > 1$, then $SW_X(\pm c_1(X,\omega)) = \pm 1$.
\end{thm}

\begin{thm}
\label{thm:Tmain}
\cite{Ta1,Ta2}, also see \cite{Ko1} (and \cite{GS}, chapter 10, for this simpler statement). Suppose that $(X,\omega)$ is a symplectic $4$-manifold with $b_2^+(X) > 1$ and $SW_X(K) \neq 0$ for a given $K \in \mathcal{C}_X$. Assume furthermore that the class $c = \frac{1}{2}(K - c_1(X,\omega))$ is nonzero in $H^2(X;\Z)$. Then for a generic compatible almost-complex structure $J$ on $X$, the class $PD(c) \in H_2(X;\Z)$ can be represented by a pseudo-holomorphic submanifold (not necessarily connected).
\end{thm}

From the above result, one can also conclude the following:

\begin{thm}
\label{thm:Tmin}
\cite{Ta2,Ta4,Ko1}, also see \cite{GS}. If $X$ is a minimal symplectic $4$-manifold (i.e. does not contain symplectic spheres with self-intersection $(-1)$) with $b_2^+(X) > 1$, then $c_1^2(X,\omega) \geq 0$.
\end{thm}

From the above two results and the generalized adjunction formula, we can further conclude the following:

\begin{cor}
\label{cor:taubes}
\cite{Ta4}, also see \cite{GS}. If $(X,\omega)$ is a symplectic $4$-manifold with $c_1^2(X,\omega) \leq -1$, then for a generic compatible almost-complex structure $J$ on $X$, there exists a $J$-holomorphic sphere of self-intersection $(-1)$.
\end{cor}

\begin{proof}
From Theorem~\ref{thm:Tmin} it follows that if $c_1^2(X,\omega) \leq -1$, then there exists a symplectic sphere, $\Sigma \in X$, with $[\Sigma]^2 = -1$. However, since for the homology class $\pm [\Sigma]$ we have $c_1(X,\omega) \cdot PD([\Sigma]) = 1$ and $SW_X(c_1(X,\omega) + 2PD([\Sigma])) \neq 0$, meaning that $c_1(X,\omega) + 2PD([\Sigma]) \in \mathcal{B}as_X$, then from Theorem~\ref{thm:Tmain} we have that the homology class $[\Sigma]$ can be represented by a pseudo-holomorphic submanifold. Finally, the generalized adjunction formula forces the pseudo-holomorphic submanifold to be a sphere. \end{proof}

\subsection{Toric and almost-toric fibrations of symplectic 4-manifolds}
\label{sec:toric}

In this section we introduce toric and almost-toric models of symplectic 4-manifolds, which will be used in \textit{Step 3} (section~\ref{sec:step3}) of the proof of the main theorem. The goal is to introduce enough terminology, so that we can present the almost-toric models of manifolds $C_n$ and $B_n$, as well as illustrate how to see the ``core" $\L_n$ of $B_n$ (see section~\ref{sec:bn}) in these models.

In \cite{Sym1}, Symington showed that the rational blow-down construction can be performed in the symplectic category. She did this by describing the symplectic structure of $C_n$ and a collar neighborhood of $\partial B_n$ with the help of toric fibrations. In \cite{Sym2}, she generalized this construction to show that the generalized rational blow-down can also be performed in the symplectic category. In \cite{Sym3}, she presented a way of describing symplectic $4$-manifolds through almost-toric fibrations and used this to prove the existence of the  symplectic rational blow-down in a less cumbersome manner than using just toric fibrations. 

The goal of toric and almost-toric fibrations of symplectic $4$-manifolds is to be able to depict various topological and symplectic properties of these manifolds with polytopes and curves in plane. The basis for doing this comes from a theorem of Delzant:

\begin{thm}
\label{thm:delzant}
\cite{De} If a closed symplectic manifold $(M^{2n},\omega)$ is equipped with an effective Hamiltonian $n$-torus action, then the image of the moment map $\Delta$ determines the manifold $M$, its symplectic structure $\omega$ and the torus action.
\end{thm}

\noindent Additionally, we have the following key result on Hamiltonian torus actions:

\begin{thm}
\label{thm:polytope}
\cite{At,GuSt} The moment map image $\Delta$ for a Hamiltonian $k$-torus action on a closed symplectic manifold $(M,\omega)$ is a convex polytope.
\end{thm}

\noindent When $k=n$, the manifold $(M^{2n},\omega)$ is called toric. For our purposes we will only be dealing with the case $n=2$, and while several of the following results hold in any even dimension, we will only state them for $n=2$. The main goal of \cite{Sym3}, with the almost-toric fibrations is to extend the above two theorems to work for a larger class of symplectic $4$-manifolds, and generalize the class of moment-map images. 

Since the symplectic form vanishes on the fibers of a moment map, implying that the regular fibers are Lagrangian submanifolds, the moment map actually provides us with a Lagrangian fibration:

\begin{defn}
\cite{Sym3} A projection $\pi : (M^4,\omega) \rightarrow B^2$ is a \textit{Lagrangian fibration} if it restricts to a regular Lagrangian fibration (locally trivial fibration where the fibers are Lagrangian) over an open dense set $B_0 \subset B$.
\end{defn}

\noindent The most basic example is $\pi:(\R^2 \times T^2,\omega_0) \rightarrow \R^2$, with $\omega_0$ the standard symplectic structure, which serves as a model for all other examples. The goal is to make use of the standard lattice $\Lambda_0$ on the tangent bundle $T\R^2$, spanned by $\left\{\frac{\partial}{\partial p_i}\right\}$ and $\left\{\frac{\partial}{\partial q_i}\right\}$, where $(p,q)$ are the standard coordinates on $\R^2 \times T^2$. In relation to this, Symington shows the following:

\begin{thm}
\cite{Sym3} If $\pi: (M,\omega) \rightarrow B$ is a regular Lagrangian fibration then there are lattices $\Lambda \subset TB$, $\Lambda^{\ast} \subset T^{\ast}B$ and $\Lambda^{vert}$ in the vertical bundle of $TM$ (induced by $\pi$) that, with respect to standard local coordinates, are the standard lattice, its dual, and the standard vertical lattice.
\end{thm}

\noindent This induced lattice $\Lambda$ on the tangent bundle of the base $B$, as above, gives $B$ an \textit{integral affine structure} $\A$.

\begin{prop}
\cite{Sym3} An $n$-manifold $B$ admits an integral affine structure if and only if it can be covered by coordinate charts $\left\{U_i,h_i \right\}$, $h_i : U_i \rightarrow \R^n$ such that the map $h_j \circ h_i^{-1}$, wherever defined, is an element of $AGL(n,\Z)$, i.e. a map of the form $\Phi(x) = Ax + b$ where $A \in GL(n,\Z)$ and $b \in \R^n$.
\end{prop}

Symington denotes the toric (and almost-toric) bases with $(B,\A,\S)$, where $B$ is the polytope base in $\R^n$ (see Theorem~\ref{thm:polytope}), $\A$ is an \textit{integral affine structure}, and $\S$ is a natural stratification of the base $B$: the $l$-stratum is the set of points $b \in B$ such that $\pi^{-1}(b)$ is a torus of dimension $l$. Additionally, $\partial_R B$  denotes the collection of all the $k$-strata, with $k < n$, which is the \textit{reduced boundary} of the base $(B,\A,\S)$. Symington gives the following definition of the \textit{toric fibration and base}:

\begin{defn}
\cite{Sym3} A Lagrangian fibration $\pi: (M^4,\omega) \rightarrow (B,\A,\S)$ is a \textit{toric fibration} if there is a Hamiltonian $2$-torus action and an immersion $\Phi: (B,\A) \rightarrow (\R^2,\A_0)$ such that $\Phi \circ \pi$ is the corresponding moment map and $\S$ is the induced stratification. In this case we call $(B,\A, \S)$ a \textit{toric base}.
\end{defn}

Since we are looking to represent symplectic $4$-manifolds, we will be working with bases of dimension $2$, and with $2$, $1$ and $0$-strata. In other words, the $1$-stratum are the edges of our polytope $B$ in the plane, and the $0$-stratum are its vertices. Consequently, Symington's goal was to put the appropriate conditions on the base $(B,\A,\S)$ to ensure that it determines a unique symplectic $4$-manifold. To reconstruct a symplectic $4$-manifold from a toric base $(B,\A,\S)$, one can start with a regular Lagrangian fibration over $(B,\A)$ and collapse certain fibers to get the desired stratification $\S$. Symington does this with the help of \textit{symplectic boundary reduction}, introduced in \cite{Sym1}, which is defined in the proposition below:

\begin{prop}
Let $(M,\omega)$ be a symplectic manifold with boundary such that a smooth component $Y$ of $\partial M$ is a circle bundle over a manifold $\Sigma$. Suppose also that the tangent vectors to the circle fibers lie in the kernel of $\omega|_{Y}$. Then there is a projection $\rho : (M,\omega) \rightarrow (M',\omega')$ and an embedding $\phi: \Sigma \rightarrow M'$ such that $\rho(Y) = \phi(\Sigma)$, $\rho|_{M-Y}$ is a symplectomorphism onto $M' - \phi(\Sigma)$ and $\phi(\Sigma)$ is a symplectic submanifold. The manifold $(M',\omega') = \rho(M,\omega)$ is the \textbf{symplectic boundary reduction} of $(M,\omega)$ along $Y$.
\end{prop}

\noindent Connecting the above proposition to the toric bases $(B,\A,\S)$, Symington gives the following definition:

\begin{defn}
Given a toric fibration $\pi: (M^4,\omega) \rightarrow (B,\A,\S)$, the \textit{boundary recovery} is the unique Lagrangian fibered manifold $(B \times T^2, \omega_0)$ that yields $(M,\omega)$ via boundary reduction.
\end{defn}

\begin{figure}[ht!]
\centering
\includegraphics[scale = 1.2]{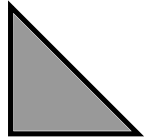}
\caption{Toric model of $\C P^2$}
\label{f:toriccp2}
\end{figure}

\begin{example}
A basic example is the toric base for a symplectic $4$-manifold diffeomorphic to $\C P^2$, which is a simple triangle with vertices on $(0,0)$, $(1,0)$ and $(0,1)$, as depicted in Figure~\ref{f:toriccp2}, with the bold edges representing the $1$-stratum. This base represents $\C P^2$ as the boundary reduction of $(B^4,\omega_0)$, where the circles of the Hopf fibration are collapsed.
\end{example}

In reading such diagrams, it is important to remember that the pre-image of each interior point in the diagram, is a torus, $S^1 \times S^1$, the pre-image of each point on the thick edges of the diagram (the 1-stratum) is a circle $S^1$, and the pre-image of each vertex in the diagram is just a point. Before describing the toric model of $C_n$, we first introduce the following important element of toric bases: 

\begin{defn}
\cite{Sym3} Let $\pi: (M,\omega) \rightarrow (B,\A,S)$ be a toric fibration and $\gamma$ a compact embedded curve with one endpoint $b_1$ in the $1$-stratum of $\partial_R B$ (both the $1$ and $0$-stratum in this case) and such that $\gamma - \left\{b_1\right\} \subset B_0 = B - \partial_R B$. Let $b_0$ be the other endpoint of $\gamma$. The \textit{collapsing class}, with respect to $\gamma$, for the smooth component of $\partial_R B$ containing $b_1$ is the primitive class $\mathbf{a} \in H_1(F_{b_0}; \Z)$ that spans the kernel of $\iota_{\ast} : H_1(F_{b_0}; \Z) \rightarrow H_1(\pi^{-1}(\gamma);\Z)$, where $\iota$ is the inclusion map. Corresponding to the \textit{collapsing class} is the \textit{collapsing covector}, with respect to $\gamma$, which is the primitive covector $v^{\ast} \in T^{\ast}_{b_0} B$ that determines vectors $v(x) \in T^{vert}_x M$ for each $x \in \pi^{-1} b_0$ such that the integral curves of this vector field represent $\mathbf{a}$.
\end{defn} 

\begin{figure}[ht!]
\labellist
\small\hair 2pt
\pinlabel $S_1$ at 27 3
\pinlabel $S_2$ at 50 7
\pinlabel $S_3$ at 62 10
\pinlabel $S_{n-1}$ at 77 23
\pinlabel $L(n^2,n-1)$ at 48 61
\endlabellist
\centering
\includegraphics[height = 60mm, width = 130mm]{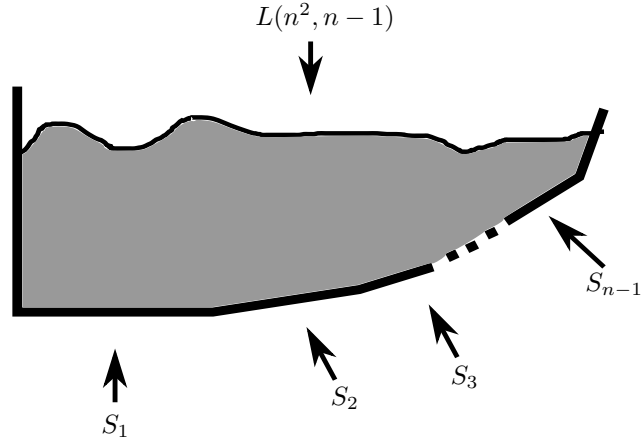}
\caption{Toric model for $C_n$}
\label{f:toriccn}
\end{figure}

\begin{example}
\label{ex:toriccn}
We can construct a toric fibration of the $C_n$ configuration of spheres, (see Figure~\ref{f:toriccn}). In this diagram, the slopes of the edges are $0, \frac{1}{n+2}, \frac{2}{2n+3}, \frac{3}{3n+4}, \ldots, \frac{n-1}{n^2}$, thus the corresponding collapsing covectors are: $v_1 = \left[\begin{array}{c}0\\1\end{array}\right], v_2 = \left[\begin{array}{c}-1\\n+2\end{array}\right], v_3 = \left[\begin{array}{c}-2\\2n+3\end{array}\right],  v_4 = \left[\begin{array}{c}-3\\3n+4\end{array}\right], \ldots,  v_n = \left[\begin{array}{c}1-n\\n^2\end{array}\right]$. Consequently, we have $v_{i+1} \times v_{i-1} = [S_i]^2$, giving us the desired self-intersection numbers of the spheres $S_i$.  

The pre-image of the (thin) curve on the top of the diagram is the boundary $\partial C_n = L(n^2,n-1)$, since the collapsing covectors on both endpoints of the curve are $v_0 = \left[\begin{array}{c}1\\0\end{array}\right]$ and $v_n = \left[\begin{array}{c}1-n\\n^2\end{array}\right]$.
\end{example}

Symington proves (Theorem 3.19, \cite{Sym3}) that such diagrams of toric bases $(B,\A,\S)$, determine unique toric manifolds, presented as the boundary reduction of $(B \times T^2,\omega_0)$ (assuming certain technical conditions, see \cite{Sym3} for details). The uniqueness of the toric manifold fibering over the base $(B,\A,\S)$, was shown earlier by \cite{BoMo}.

One can push these diagrams further, to depict Lagrangian fibrations with (nodal) singularities:

\begin{defn}
\cite{Sym3} A nondegenerate Lagrangian fibration $\pi: (M,\omega) \rightarrow B$ of a symplectic $4$-manifold is an \textit{almost-toric fibration} if it is a nondegenerate topologically stable fibration with no hyperbolic singularities (e.g. a fibration with a nodal singularity). A triple $(B,\A,\S)$ is an \textit{almost-toric base} if it is the base of such a fibration. A symplectic $4$-manifold equipped with such a fibration is an \textit{almost-toric manifold}.
\end{defn}

\noindent Thus if $\left\{s_i\right\} \subset B$ are the images of such singularities, then $\A$ is the affine structure on $B - \left\{s_i\right\}$. Generally, a Lagrangian fibration can be arranged such that nodal singularities occur in distinct fibers. Also, a nodal fiber is the singular fiber of a Lefschetz fibration, and its neighborhood is diffeomorphic to $T^2 \times D^2$ with a $(-1)$-framed two-handle attached along a simple closed curve in $T^2 \times \left\{x\right\}$. One can also compute the topological monodromy around the nodal fiber with respect to the basis $\left\{[\gamma_1],[\gamma_2]\right\} \in H_1(F_b;\Z)$:

$$
\begin{array}{c} \Psi(\gamma) \end{array}  = \begin{array}{c} A_{(1,0)} \end{array} = \left(\begin{array}{cc} 1 & 1 \\ 0 & 1\end{array}\right) \, . 
$$

\noindent If we choose a different basis for $H_1(F_b;\Z)$, then we conjugate the matrix $A_{(1,0)}$, giving us the following monodromy matrix with eigenvector $(a,c)$:

$$
\begin{array}{c} A_{(a,c)} \end{array} = \left(\begin{array}{cc} 1-ac & a^2 \\ -c^2 & 1+ac\end{array}\right) \, .
$$ 

\noindent This leads us to the following definition and lemma:

\begin{defn}
\cite{Sym3} Let $\pi:(M,\omega) \rightarrow B$ be an almost-toric fibration with a node at $s$. Let $\eta$ be an embedded curve with endpoints at $s$ and a point $b \in B_0 = B - \partial_R B$ such that $\eta - \left\{s\right\} \subset B_0$ contains no other nodes. A \textit{vanishing class}  in $H_1(F_b;\Z)$, associated to $s$ and $\eta$, is the class whose representatives bound a disk in $\pi^{-1}(\eta)$. The \textit{vanishing covector} $w^{\ast} \in T^{\ast}_b B$ is the primitive covector that determines vectors $w(x) \in T^{vert}_x M$ for each $x \in \pi^{-1}b$ such that the integral curves of this vector field represent the vanishing class. 
\end{defn}

\begin{lemma}
\cite{Sym3} Suppose $\gamma$ is a positively oriented loop based at $b$ that is the boundary of a closed neighborhood of $s$ containing $\eta$. Then the vanishing class is the unique class (up to scale) that is preserved by the monodromy along $\gamma$. With respect to the basis for $H_1(F_b;\Z)$ for which the monodromy matrix is $A_{(a,c)}$, the vanishing class is the class $(a,c)$.
\end{lemma}

\begin{figure}[ht!]
\centering
\includegraphics[height = 50mm, width = 50mm]{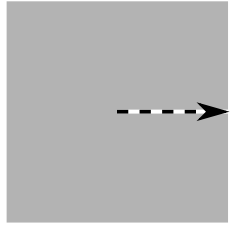}
\caption{Almost toric base}
\label{f:toricray}
\end{figure}

Notice, that given such an almost-toric fibration $\pi: (M,\omega) \rightarrow B$ with a node at $s$, the fibration over $B- \left\{s\right\}$ is regular and has an induced affine structure $\A$. However, there is non-trivial monodromy around the node $s$, thus there is no affine immersion of $(B - \left\{s\right\},\A)$ into $(\R^2,\A_0)$. To salvage this, we can remove a ray $R$, based at the node $s$, from the base $B$, giving us an immersion of $(B-R,\A)$ into $(\R^2,\A_0)$. An example of such a base $B$ with a removed ray $R$ is seen in Figure~\ref{f:toricray}. This ray (or eigenray) is an eigenvector of the monodromy matrix of the node. Symington shows, that any integral affine punctured plane $(V,\A)$ will be isomorphic to the one depicted in Figure~\ref{f:toricray}, where the ray $R$ is the eigenray $(1,0)$. Consequently, we can model an almost-toric manifold with bases $B$ containing nodes $s_i$:

\begin{defn}
\cite{Sym3} An \textit{integral affine manifold with nodes} $(B,\A)$ is a two-manifold $B$ equipped with an integral affine structure on $B - \left\{s_i\right\}$ such that each $s_i$ has a neighborhood $U_i$ such that $(U_i - s_i, \A)$ is affine isomorphic to a neighborhood of the puncture in $(V^k,\A^k)$ (if the node has multiplicity $k$).
\end{defn}

\begin{thm}
\cite{Sym3} Consider a triple $(B,\A,\S)$ such that $(B,\A)$ is an integral affine manifold with nodes $\left\{s_i\right\}^N_{i=1}$. Then $(B,\A,\S)$ is an almost-toric base if and only if every point in $B - \left\{s_i\right\}^N_{i=1}$ has a neighborhood that is a toric base.
\end{thm}

Symington also defined various operations, like the \textit{nodal slide} and the \textit{nodal trade} to get from an almost toric base $(B,\A,\S)$ to another $(B',\A ', \S ')$, with both representing the same manifold with isotopic symplectic structures. Now we are ready to describe the almost-toric base for the rational homology balls $B_n$:

\begin{figure}[ht!]
\labellist
\small\hair 2pt
\pinlabel $s$ at 28 42
\pinlabel $\text{slope}\,\frac{1}{n}$ at 33 8
\pinlabel $\text{slope}\,\frac{n-1}{n^2}$ at 80 27
\endlabellist
\centering
\includegraphics[height = 60mm, width = 80mm]{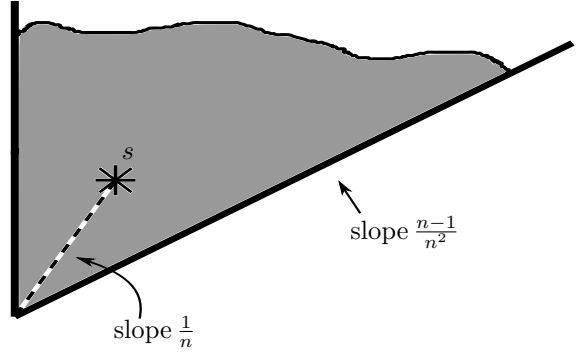}
\caption{Almost toric base for $B_n$}
\label{f:toricbn}
\end{figure}

\begin{example}
Figure~\ref{f:toricbn} depicts an almost-toric base for the rational homology balls $B_n$, in this diagram, the ray $R$ has a of slope of $\frac{1}{n}$, corresponding to the eigenvector $(n,1)$ of the monodromy 

$$
\begin{array}{c} A_{(n,1)} \end{array} = \left(\begin{array}{cc} 1-n & n^2 \\ -1 & 1+n\end{array}\right) \, .
$$

\noindent thus making $\left[\begin{array}{c} -1 \\ n\end{array}\right]$ be the vanishing covector of the node $s$. The slope of the line on the right is $\frac{n-1}{n^2}$, therefore the preimage of the thin line on the top of the diagram is $L(n^2,n-1)$ as was the case for the toric diagram for $C_n$.
\end{example}

Symington then proves that the symplectic rational blow-down can be performed in the symplectic category, by simply removing the images of the neighborhoods of the symplectic spheres from the toric model  of $C_n$ (Figure~\ref{f:toriccn}) and gluing below it, the almost-toric model for $B_n$ (Figure~\ref{f:toricbn}). They match up, since the slopes of the right-most edge is $\frac{n-1}{n^2}$, as illustrated in Figure~\ref{f:toricrbd}.

\begin{figure}[ht!]
\labellist
\small\hair 2pt
\pinlabel $L(n^2,n-1)$ at 60 100
\pinlabel $S_1$ at 33 56
\pinlabel $S_2$ at 55 58
\pinlabel $S_{n-1}$ at 103 60
\pinlabel $\text{remove}$ at 80 30
\pinlabel $L(n^2,n-1)$ at 220 145
\pinlabel $s$ at 190 45
\pinlabel $\text{slope}\,\frac{n-1}{n^2}$ at 285 110
\pinlabel $\text{slope}\,\frac{n-1}{n^2}$ at 240 30
\pinlabel $glue$ at 225 95
\pinlabel $L(n^2,n-1)$ at 370 120
\pinlabel $\text{slope}\,\frac{n-1}{n^2}$ at 400 50
\pinlabel $s$ at 348 65
\endlabellist
\centering
\includegraphics[height = 50mm, width = 130mm]{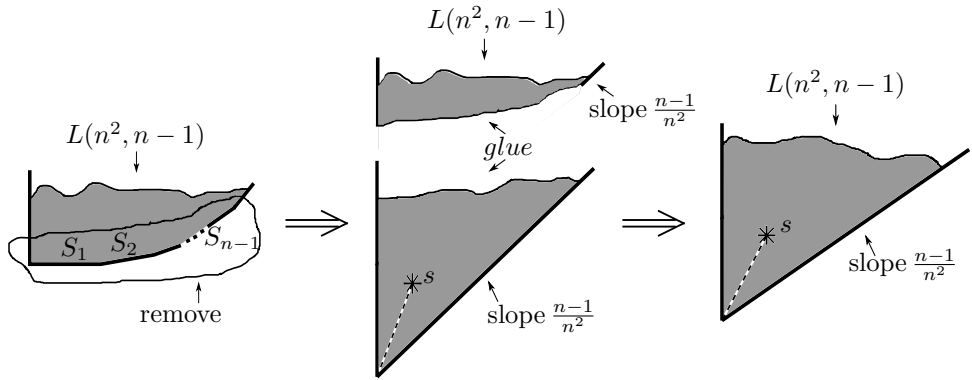}
\caption{Rational blow-down in almost-toric diagrams}
\label{f:toricrbd}
\end{figure} 

It is useful for our purposes to illustrate where on this almost-toric model of $B_n$ can we ``see" the image of the ``Lagrangian cores" $\L_n$ of the rational homology $B_n$ (see section~\ref{sec:bn} and \cite{Kh2}). Before we do this, we must first introduce the concept of \textit{visible surfaces} in these almost-toric fibrations, as was done in \cite{Sym3}.

If one draws a curve $\nu$ in an (almost)-toric base $B$, then the pre-image of every point $b \subset B_0$ in the curve will be a torus $F_b \cong S^1 \times S^1$. However, if for every point $b$ in the curve we choose a closed curve in $F_b \cong S^1 \times S^1$, then the entire collection of those closed curves over all points in the curve $\nu$ could potentially be a surface in the original $4$-manifold. This is precisely what \textit{visible surfaces} are, they are a coherent collection of such closed curves, in the pre-images of the points in a toric (almost-toric) base. Here is a more precise definition that Symington gives:

\begin{defn}
\label{def:vissurf}
\cite{Sym3} A \textit{visible surface} $\Sigma_{\nu}$ in an almost-toric fibered manifold $\pi: (M,\omega) \rightarrow (B,\A, \S)$ is an immersed surface whose image is an immersed (connected) curve $\nu$ with transverse self-intersections such that $\pi |_{\Sigma_{\nu} \cap \pi_{-1}(B_0)}$ is a submersion onto $\nu \cap B_0$, any non-empty intersection of $\Sigma_{\nu}$ with a regular fiber is a union of affine circles, and no component of $\partial \Sigma_{\nu}$ projects to a node. 
\end{defn}

The following are the conditions on curves in the base to represent a \textit{visible surface} and for a curve $\nu$ to represent a unique surface $\Sigma_{\nu}$:

\begin{defn}
\label{d:vissurf}
\cite{Sym3} Given an immersed curve $\nu : I \rightarrow (B,\A, S)$, let $\left\{\nu_i\right\}^k_{i=1}$ be the continuous (and connected) components of $\nu |_{\nu^{-1}(B-\partial_R B)}$. A primitive class $\mathbf{a}_i$ in $H_1(\pi^{-1}(\nu_i);\Z)$ (such that $\pi_{\ast}\mathbf{a}_i = 0$ if $\nu_i$ is a loop) is \textit{compatible} with $\nu$ if all of the following are satisfied:
\begin{enumerate}
\item
$\mathbf{a}_i$ is the \textit{vanishing class} of every node in $\nu$,

\item
$|\mathbf{a}_i \cdot \mathbf{c}| \in \left\{0,1\right\}$ for each $\mathbf{c}$ that is the collapsing class, with respect to $\overline{\nu_i}$, for a component of the $1$-stratum of $\partial_R B$ that intersects $\overline{\nu_i}$,

\item
$|\mathbf{a}_i \cdot \mathbf{c}|=1$ if $\overline{\nu_i}$ intersects the $1$-stratum non-transversally,

\item
$|\mathbf{a}_i \cdot \mathbf{d}|=1$ for each $\mathbf{d}$ that is one of the two collapsing classes at a vertex contained in the closure of $\nu_i$. (Here, $\cdot$ is the intersection pairing in $H_1(\pi^{-1}(\nu_i), \Z)$ and $\overline{\nu_i}$ is the closure of $\nu_i$.)
\end{enumerate}
\end{defn}

\begin{thm}
\cite{Sym3} Suppose $(B,\A,\S)$ is an almost-toric base such that each node has multiplicity one. An immersed curve $\nu: I \rightarrow (B,\A,\S)$ with transverse self-intersections and a set of compatible classes $\left\{\mathbf{a}_i\right\}^k_{i=1}$ together determine a visible surface $\Sigma_{\nu}$ such that for each $b \in \nu_i$, 
\begin{equation}
\label{eq:iota}
\iota_{\ast}[\Sigma_{\nu} \cap F_b] = \mathbf{a}_i
\end{equation}

\noindent where $\iota : F_b \rightarrow \pi^{-1}(\nu_i)$ is the inclusion map. (Note, we will not define the ``multiplicity" of a node here; all of the nodes that we will work with have ``mutliplicity" one, for details see \cite{Sym3}.) The surface $\Sigma_{\nu}$ is unique up to isotopy among visible surfaces in the preimage of $\nu$ that satisfy equation~\ref{eq:iota}. Furthermore, no such surface exists if the classes $\mathbf{a}_i$ are not compatible with $\nu$.
\end{thm}

To each primitive class $\mathbf{a}_i \in H_1(\pi^{-1}(\nu_i);\Z)$ there is corresponding \textit{compatible vector} $v_i \in \R^2$ such that the integral curves of the vector field $v_i \frac{\partial}{\partial q} \subset \Lambda^{vert}$ represent $\mathbf{a}_i$. If $v$ and $w$ are compatible vectors for primitive classes $\mathbf{a}$ and $\mathbf{b}$ respectively, then $|\mathbf{a} \cdot \mathbf{b} | = |v \times w| = |det(vw)|$. Symington also shows that if curves $\nu_1$ and $\nu_2$ intersect transversally at a point $b \in B_0$ and $\Sigma_{\nu_1}$ and $\Sigma_{\nu_2}$ intersect transversally in $F_b$, then $\Sigma_{\nu_1}$ intersects $\Sigma_{\nu_2}$ in $|v_1 \times v_2|$ points where the signs of all intersections is $det(u_1 u_2) det(v_1 v_2)$. Here, $v_i$ are the compatible vectors of $\nu_i$ and the $u_i$ are the tangent vectors of $\nu_i$ at the point $b$.

In \cite{Sym3}, it is proved that one can compute the symplectic area of the \text{visible surfaces} as follows:

\begin{prop}
\label{p:vissurfarea}
Let $\nu : I \rightarrow (B,\A,\S)$ be a parameterized immersed curve and $\left\{v_i\right\}^N_{i=1}$ a set of co-oriented compatible vectors in a base diagram that define an oriented surface $\Sigma_{\nu}$. The the (signed) area of $\Sigma_{\nu}$ is:
\begin{equation}
\label{eq:vissurfarea}
\text{Area}(\Sigma_{\nu}) = \int_{\Sigma_{\nu}} \omega = 2\pi \int^1_0 \nu'(t) \cdot v(t)dt
\end{equation}

\noindent where $v(t) = v_i$, if $\nu(t) \in \nu_i$ and for other values of $t$ (when $\nu \subset \partial_R B$) $v(t)$ is an integral vector such that $u(t) \times v(t) = 1$ for some integral vector $u(t) = \lambda\nu'(t)$, $\lambda > 0$.
\end{prop}

\begin{figure}[ht!]
\labellist
\small\hair 2pt
\pinlabel $s$ at 28 42
\pinlabel $\upsilon$ at 20 43
\endlabellist
\centering
\includegraphics[height = 60mm, width = 80mm]{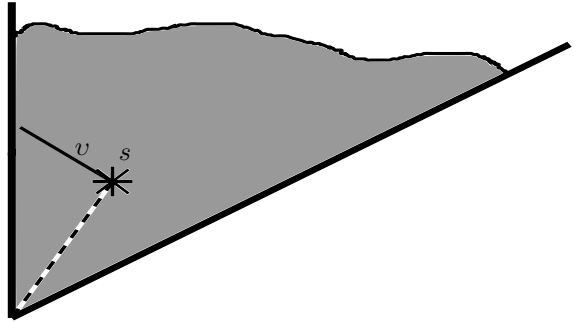}
\caption{``visible" $\L_n$ in almost toric base for $B_n$}
\label{f:toriccore}
\end{figure}

\begin{rmk}
\label{r:toriccore}
Note, that given such conditions for a visible surface, $\Sigma_{\nu}$ must be a sphere, disk, cylinder or torus. Therefore, if we want to ``see" a Lagrangian core $\L_n$ in the almost-toric base for $B_n$, we can only really ``see" where $\L_n$ is an embedding, in other words, a Lagrangian disk in $\L_n$, right before the edge of the disk hits the singular part of $\L_n$. Proposition~\ref{p:vissurfarea} implies that in order for a visible surface $\Sigma_{\upsilon}$ to be Lagrangian, we must have that $\upsilon$ is a straight line. The line $\upsilon$ in Figure~\ref{f:toriccore}, extending from the node and (almost) hitting the left edge of the $1$-stratum, represents a Lagrangian visible surface $\Sigma_{\upsilon}$. Since the line $\upsilon$ hits a node, its compatible covector must correspond to the vanishing covector of the node, which is $v = \left[\begin{array}{c} -1 \\ n\end{array}\right]$. Notice, if $\upsilon$ were to actually hit the left edge of the $1$-stratum, then this would violate condition $(2)$ of Definition~\ref{d:vissurf}, since $|v \times c| = n$, where $c$ is the collapsing covector of the left edge of the $1$-stratum. As a result, $\Sigma_{\upsilon}$ can represent the Lagrangian core $\L_n$ of $B_n$, as introduced in section~\ref{sec:bn}, since the boundary of the 2-cell $D^2$ in $\L_n$ wraps around $n$ times the $S^1$ in $\L_n$.
\end{rmk}

\section{Proof of Main Theorem}
\label{sec:sympemb}

We now again present the statement of the main theorem on symplectic embeddings of $B_n$, which appears in the introduction, preceded by the following crucial proposition and some definitions and terminology.

\begin{prop}
\label{p:sympemb}
Let $(X,\omega)$ be a symplectic $4$-manifold, such that:
\begin{itemize}
\item
$b_2^+(X) > 1$,
\item
$\left[c_1(X,\omega)\right] = -\left[\omega \right]$ as cohomology classes and
\item
$n \geq c_1^2(X,\omega) + 2$.
\end{itemize}

\noindent If there exists a symplectic embedding $B_n \hookrightarrow (X,\omega)$ and $(X',\omega')$ is the symplectic rational blow-up of $(X,\omega)$, then there exists an embedded symplectic sphere $\Sigma_{-1} \subset (X',\omega')$, and a linear plumbing configuration $C_n \subset (X',\omega')$ of symplectic spheres $S_j$, $1 \leq j \leq n-1$, such that:
\begin{itemize}
\item
$[\Sigma_{-1}]^2= -1$,
\item
$[S_1]^2 = -n-2$ and $[S_j]^2 = -2$ for $ 2 \leq j \leq n-1$ (see Figure~\ref{f:cn}) and
\item
$\Sigma_{-1}$ intersects the spheres $S_j$, $1 \leq j \leq n-1$ positively and transversally.
\end{itemize}

\end{prop}

\begin{defn}
\label{d:bntype}
We call a symplectic embedding of $B_n \hookrightarrow (X,\omega)$ to be of \textit{type} $\langle \alpha_1, \alpha_2, \alpha_3, \ldots , \alpha_{n-1} \rangle$, where $\alpha_j \in \Z_{\geq 0}$, if there exists an embedded symplectic sphere, $\Sigma \subset X'$, with $[\Sigma]^2 = -1$, such that it intersects positively and transversally with the spheres $S_j$, $1 \leq j \leq n-1$, of the $C_n$ configuration in $X'$ and $\alpha_j$ is the number of those positive transverse intersections. 
\end{defn}

\begin{defn}
\label{d:bntypeA}
Let $\A$ be the set of $(n-1)$-tuples $\langle \alpha_1, \alpha_2, \alpha_3, \ldots , \alpha_{n-1} \rangle$ such that:
\begin{enumerate}
\item
$\alpha_j \neq 0$ for at least one $j$, where $2 \leq j \leq n-1$, or

\item
$\alpha_1 \geq n$, or

\item
$\alpha_1 = 1$ and $\alpha_j =0$ for $2 \leq j \leq n-1$.
\end{enumerate}
We will call a symplectic embedding $B_n \hookrightarrow X$ to be of \textit{type} $\A$ if it is of type $\langle \alpha_1, \alpha_2, \alpha_3, \ldots, \alpha_{n-1} \rangle \subset \A$.
\end{defn}

\begin{defn}
\label{d:bntypeEk}
Let $\E_k$ denote the $(n-1)$-tuple $\langle k, 0, 0, \ldots, 0 \rangle$ for $2 \leq k \leq n-1$. 
\end{defn}

We note that Proposition~\ref{p:sympemb} implies that a symplectic embedding $B_n \hookrightarrow (X,\omega)$ (for $b_2^+(X) > 1$, $\left[c_1(X,\omega)\right] = -\left[\omega \right]$ and $n \geq c_1^2(X,\omega) + 2$) will always be of type $\langle \alpha_1, \alpha_2, \alpha_3, \ldots, \alpha_{n-1} \rangle$, for some $(n-1)$-tuple $\langle \alpha_1, \alpha_2, \alpha_3, \ldots, \alpha_{n-1} \rangle$ with $\alpha_j \in \Z_{\geq 0}$ . Moreover, any $(n-1)$-tuple $\langle \alpha_1, \alpha_2, \alpha_3, \ldots , \alpha_{n-1} \rangle$ with $\alpha_j \in \Z_{\geq 0}$ will be in at least one of the sets $\A, \E_k$, $2 \leq k \leq n-1$.

\begin{thm} 
\label{thm:sympemb}
If $B_n \hookrightarrow (X,\omega)$ is a symplectic embedding, where  $(X,\omega)$ is a symplectic $4$-manifold, such that:
\begin{itemize}
\item
$b_2^+(X) > 1$,
\item
$\left[c_1(X,\omega)\right] = -\left[\omega \right]$ as cohomology classes,
\item
$n \geq c_1^2(X,\omega) + 2$ and
\item
$\mathcal{B}as_X = \left\{\pm c_1(X,\omega)\right\}$, ($\mathcal{B}as_X$ denotes the set of Seiberg-Witten basic classes of X,)
\end{itemize}
then it cannot be of type $\A$ or of type $\E_k$, $k \geq c_1^2(X,\omega) + 2$.
\end{thm}

\begin{rmk}
The condition $\left[c_1(X,\omega)\right] = -\left[\omega \right]$ holds for surfaces of general type $X$, with the canonical class $K_X$ ample. The ampleness implies that for all curves $C$ in $X$, we have that $c_1(X,\omega) \cdot [C] < 0$, implying that there are no $(-1)$ or $(-2)$ curves in $X$. For a symplectic $4$-manifold $X$, the condition $\left[c_1(X,\omega)\right] = -\left[\omega \right]$ implies that there are no symplectic spheres $S$ with self-intersection $(-1)$ or $(-2)$: since for a symplectic sphere $S$, we have $\int_S \omega > 0$ which implies $c_1(X) \cdot [S] < 0 \Rightarrow [S]^2 < -2$ by the adjunction inequality. Additionally, the condition of $(X,\omega)$ having only one Seiberg-Witten basic class (up to sign), is also true of all surfaces of general type. Consequently, these symplectic $4$-manifolds are meant to mimic surfaces of general type as much as they can.
\end{rmk}

We will prove this theorem in four steps. In \textit{\textbf{Step 1}}, section~\ref{sec:step1}, we will prove Proposition~\ref{p:sympemb}. In \textit{\textbf{Step 2}}, section~\ref{sec:step2}, using the existence of the sphere $\Sigma_{-1}$ from Proposition~\ref{p:sympemb}, we construct a specific homology cycle $\gamma$, and compute $c_1(X,\omega) \cdot \gamma$ in terms of the intersection pattern of $\Sigma_{-1}$ with the spheres of the $C_n$ configuration. In \textit{\textbf{Step 3}}, section~\ref{sec:step3}, we show that if $c_1(X,\omega) \cdot \gamma > 0$, then $\omega \cdot \gamma > 0$, thus contradicting the $\left[c_1(X,\omega)\right] = -\left[\omega \right]$ assumption on $(X,\omega)$. As a result, we will show show that $B_n \hookrightarrow (X,\omega)$ cannot be of type $\A_1 \subset \A$, where $\A_1$ is the set of $(n-1)$-tuples corresponding to $c_1(X,\omega) \cdot \gamma > 0$. In \textit{\textbf{Step 4}}, section~\ref{sec:step4}, we show that if $c_1(X,\omega) \cdot \gamma \leq 0$, then this violates certain adjunction inequalities or forces $X$ to have additional Seiberg-Witten basic classes, thus preventing $B_n \hookrightarrow (X,\omega)$ to be of type $(\A - \A_1)$ and $\E_k$, $k \geq c_1^2(X,\omega) + 2$.

In section~\ref{sec:e2}, we give explicit examples of symplectic embeddings of $B_n \hookrightarrow (X,\omega)$ of type $\E_2$ for $n$ odd. In these examples, we always have $n < 3 + \frac{4}{3}c_1^2(X,\omega)$.

\subsection{Step 1}
\label{sec:step1}
As a first step in proving Theorem~\ref{thm:sympemb}, we will prove Proposition~\ref{p:sympemb}, that is, we will show that there exists a sphere, $\Sigma_{-1}$ of self-intersection $(-1)$ which intersects the spheres of the $C_n$ configuration, positively and transversally, in the rational blow-up of $X$. 

We begin by assuming that for a given symplectic 4-manifold $(X,\omega)$, with conditions as stated in the Proposition~\ref{p:sympemb}, there is a symplectic embedding $B_n \hookrightarrow (X,\omega)$. This embedding is in the sense of the symplectic rational blow-up theorem (Theorem 3.2 in \cite{Kh2}), meaning there is a Lagrangian core $\L_n$ in $(X,\omega)$, whose neighborhood is the rational homology ball $B_n$. It follows, according to this theorem that we can perform the symplectic rational blow-up procedure, replacing $B_n$ with $C_n$, and obtain a new symplectic manifold $(X',\omega')$ which contains a symplectic copy of a $C_n$ configuration of symplectic spheres. Since we assumed that $n \geq c_1^2(X,\omega) + 2$, and since $c_1^2(X',\omega') = c_1^2(X,\omega) - (n-1)$, we have $c_1^2(X',\omega') \leq -1$. As a consequence of Corollary~\ref{cor:taubes}, for a generic compatible almost-complex structure $J_{\epsilon}$ on $X'$, there exists a $J_{\epsilon}$-holomorphic sphere, $\Sigma_{\epsilon}^{-1}$ with self-intersection number $(-1)$. 

In order to force only positive intersections between the spheres of the $C_n$ configuration and a sphere of self-intersection $(-1)$, $\Sigma_{-1}$ (derived from $\Sigma^{-1}_{\epsilon}$ as a consequence of Proposition~\ref{p:grcomp1}), we need to make the spheres of the $C_n$ configuration pseudo-holomorphic:

\begin{lemma}
\label{l:jholocn}
With $X'$ as above, there exists an $\omega$-compatible almost-complex structure $J$ on $X'$ such that all of the spheres in the $C_n$ configuration are $J$-holomorphic.
\end{lemma}

\begin{proof}
First, we label the spheres of $C_n$ with $S_1, S_2, S_3, \ldots , S_{n-1}$, as before in Figure~\ref{f:cn}. Let the points $a_i = S_i \cap S_{i+1}$ be the points in the intersection of the spheres of $C_n$. Let $N_{a_i}$ be small Darboux neighborhoods around those points, such that 
\begin{equation}
E = \bigcup_{i=1}^{n-1} S_i - \bigcup_{i=1}^{n-2}(N_{a_i} \cap (S_i \cup S_{i+1}))
\end{equation}

\noindent is a symplectic submanifold consisting of $(n-1)$ connected components. Then, we can choose an $\omega$-compatible almost-complex structure $J$ on $X'$ such that all the connected components of the submanifold $E$ are $J$-holomor\-phic submanifolds. 

We can extend this almost-complex structure $J$ across the neighborhoods of the intersection points $N_{a_i}$ as follows: First, the results of \cite{McP} imply that for the symplectic spheres in $C_n$ configuration, which intersect transversally and positively, can always be isotoped in such a way that they intersect orthogonally (with respect to the symplectic structure) while remaining symplectic. Second, we use the following technical local result, which is a version of McDuff's result in \cite{McD1}:

\begin{lemma}
\label{l:localortho}
Let $\pi_1$ and $\pi_2$ be two orthogonal planes through $\left\{0\right\}$ in $\R^4$ which intersect with positive orientation and are symplectic with respect to the standard linear symplectic form $\omega_0$. Then there is a linear $\omega_0$-compatible $J$ which preserves these planes.
\end{lemma}

\begin{proof}
We can choose a basis $(e_1,e_2)$ for $\pi_1 \subset \R^4$ and a basis $(e_3,e_4)$ for $\pi_2 \subset \R^4$, such that $\omega_0(e_1,e_2)=1$ and $\omega_0(e_3,e_4)=1$ and $\pi_1^{\bot} = \pi_2$ (with respect to $\omega_0$). Then we simply choose $J$ to be such that $J(e_1)=e_2$ and $J(e_3)=(e_4)$. \end{proof}

\noindent Since after (possibly) isotoping the symplectic spheres of $C_n$, the intersections of the spheres are orthogonal, in a local Darboux neighborhood, $N_{a_i}$, they can be modeled by two orthogonal planes through $\left\{0\right\}$ in $\R^4$. Therefore, Lemma~\ref{l:localortho} implies that we can choose an $\omega$-compatible almost-complex structure $J$ on $X'$ such that the symplectic spheres of $C_n$ are also $J$-holomorphic spheres. \end{proof}

\begin{rmk}
McDuff's result \cite{McD1}, says that if the planes $\pi_1$ and $\pi_2$ intersect positively and transversally then there exists an $\omega$-tame almost-complex structure $J$ preserving the planes. This is not enough for our purposes, since in the next step, using Gromov compactness we will consider a sequence of almost-complex structures from $J_{\epsilon} \rightarrow J$, and since $J_{\epsilon}$ is required to be $\omega$-compatible by Taubes' theorem, we need $J$ to be $\omega$-compatible as well.
\end{rmk}

\begin{prop}
\label{p:jholosphere}
Let $X'$ be the rational blow-up of $X$, as above, then there exists a $J$-holomorphic sphere of self-intersection $(-1)$ in $X'$, $\Sigma_{-1}$, with $J$ the almost-complex structure from Lemma~\ref{l:jholocn}.
\end{prop}

\begin{proof}
To show the existence of this $J$-holomorphic sphere, $\Sigma_{-1}$ we will use Gromov compactness to find a sequence of almost-complex structures, under which the $J_{\epsilon}$-holomorphic sphere $\Sigma_{\epsilon}^{-1}$ will converge to a multicurve, (or a cusp-curve) with (potentially) some ``bubbles". One of the components of the multicurve will be a $J$-holomorphic sphere of self-intersection $(-1)$, $\Sigma_{-1}$. First, we state the definition and properties of a multicurve, convergence of almost-complex structures and Gromov compactness. Let $(M,\omega)$ be a compact symplectic manifold:

\begin{defn}
\cite{McS2} A \textit{multicurve} (or \textit{cusp-curve}) $C$ is a connected union
\begin{equation}
C=C^1 \cup C^2 \cup \cdots \cup C^N
\end{equation}

\noindent of $J$-holomorphic spheres $C^j$, which are called components. Each component is parameterized by a smooth nonconstant $J$-holomorphic map $u^j: \C P^1 \rightarrow M$, which is not required to be simple. The multicurve is denoted by $u = (u^1, \ldots , u^N)$.
\end{defn}

\begin{defn}
\cite{McS2} A sequence of $J$-holomorphic curves $u_{\nu}:\C P^1$ is said to \textbf{converge weakly} to a multicurve $u = (u^1, \ldots , u^N)$ if the following holds:

\begin{enumerate}
\item
For every $j \le N$, there exists a sequence $\phi^j_{\nu}: \C P^1 \rightarrow \C P^1$ of fractional linear transformations and a finite set $X^j \subset \C P^1$ such that $u_{\nu} \circ \phi^j_{\nu}$ converges to $u^j$ uniformly with all derivates on compact subsets of $\C P^1 - X^j$.

\item
There exists a sequence of orientation preserving (but not holomorphic) diffeomorphisms $f_{\nu}: \C P^1 \rightarrow \C P^1$ such that $u_{\nu} \circ f_{\nu}$ converges in the $C^0$-topology to a parametrization $v:\C P^1 \rightarrow M$ of the multicurve $u = (u^1, \ldots , u^N)$.
\end{enumerate}
\end{defn}

It follows \cite{McS2}, that for $\nu$ sufficiently large, that the map $u_{\nu}: \C P^1 \rightarrow M$ is homotopic to:
\begin{equation}
u^1 \# u^2 \# \cdots \# u^N : \C P^1 \rightarrow M \, .
\end{equation}

\noindent In particular if $A_{\nu},A^j \in H_2(M,\Z)$ are the homology classes of $u_{\nu}$ and $u^j$ respectively, then we have:
\begin{equation}
c_1(M) \cdot A_{\nu} = \sum_{j=1}^N c_1(M) \cdot A^j \, .
\end{equation}

Finally, we can state Gromov's compactness theorem \cite{Gr}, as it appears in \cite{McS2}:

\begin{thm}
\label{thm:gromov}
\textbf{(Gromov's compactness)} Assume $M$ is compact, and let $J_{\nu} \in \J_{\tau}(M,\omega)$ be a sequence of $\omega$-tame almost complex structures which converge to $J$ in the $C^{\oo}$-topology. Then any sequence $u_{\nu}: \C P^1 \rightarrow M$ of $J_{\nu}$-holomorphic spheres with $\text{sup}_{\nu} E(u_{\nu}) < \oo$ has a subsequence which converges weakly to a (possible reducible) $J$-holomorphic multicurve $u = (u^1, \ldots , u^N)$.
\end{thm}

Additionally, specifically for symplectic manifolds of dimension $4$, we have the following adjunction formula:

\begin{thm}
\label{thm:adjformula}
Adjunction Formula (\cite{McS3} App. E). Let $(M,J)$ be an almost-complex 4-manifold, $(\Sigma, J)$ be a closed Riemann surface, not necessarily connected, and $u: \Sigma \rightarrow M$ be a simple $J$-holomorphic curve. Denote $A \in H_2(M;\Z)$ the homology represented by $u$. Then
\begin{equation}
2\delta(u) \leq A \cdot A - c_1(M) \cdot A + \chi(\Sigma)
\end{equation}

\noindent with equality if and only if $u$ is an immersion and all self-intersections are transverse.
\end{thm}

\noindent In the above, ``simple" means not multiply covered and $\delta(u)$ is the number of self-intersections of $u$:
\begin{equation}
\label{eq:deltau}
\delta(u) := \frac{1}{2} \# \left\{(z_0,z_1) \in \Sigma \times \Sigma | u(z_0)=u(z_1), z_0 \neq z_1 \right\}
\end{equation}

\noindent
Additionally, McDuff also proved the following corollary to the theorem above:

\begin{cor}
\label{cor:adjformula}
(\cite{McS3} App. E). Let $M$, $\Sigma$, $u$ and $A$ be as in Theorem~\ref{thm:adjformula}. Then 
\begin{equation}
\label{eq:jhadj}
A \cdot A - c_1(M) \cdot A + \chi(\Sigma) \geq 0
\end{equation}
\noindent with equality if and only if $u$ is an embedding.
\end{cor}

In \cite{McS3}, McDuff proves Theorem~\ref{thm:adjformula} by showing that in dimension $4$, a homology class $A \in H_2(M,\Z)$ which is represented by a simple $J$-holomorphic curve, $u: \Sigma \rightarrow M$, can always be represented by an immersed $J'$-holomorphic curve $v: \Sigma \rightarrow M$, with transverse self-intersections. The curves $u$ and $v$ are $C^1$-close, and the almost-complex structures $J$ and $J'$ are $C^1$-close as well. This is shown using a strong theorem of Micallef-White \cite{MW} which states that a singularity of a $J$-holomorphic curve is equivalent to a singularity of a holomorphic curve, up to a $C^1$-diffeomorphism. As a result, Li in \cite{Li} made the following observation:

\begin{lemma}
\label{l:ss}
If a homology class $A \in H_2(M;\Z)$, for a $4$-dimensional symplectic manifold $M$, is represented by a simple $J$-holomorphic curve $u: \Sigma \rightarrow M$ for some $\omega$-tamed almost-complex structure $J$, then $A$ is represented by an embedded symplectic surface.
\end{lemma}

In our case, the spheres of the $C_n$ configuration are $J$-holomorphic, where\-as the sphere with self-intersection $(-1)$, $\Sigma_{\epsilon}^{-1}$, is $J_{\epsilon}$-holomorphic. So by Gromov's compactness theorem, we can take a sequence of almost-complex structures $J_{\epsilon} \rightarrow J$, such that there will exist a subsequence under which the $J_{\epsilon}$-holomorphic sphere $\Sigma_{\epsilon}^{-1}$ will converge to some multicurve $u = (u^1, \ldots ,$ $u^N)$. Since the $u^i$'s can be multiply covered (multiplicity $m_i$), we will write $v^i$ for the underlying simple $J$-holomorphic curve, giving us $[u^i] = m_i[v^i]$ as homology classes in $H_2(X';\Z)$. Also, we have:
\begin{equation}
[\Sigma^{-1}_{\epsilon}] = m_1[v^1] + m_2[v^2] + \cdots + m_N[v^N]
\end{equation}

\noindent in $H_2(X';\Z)$. Next, in Proposition~\ref{p:grcomp1}, our goal is to show that one of the $v^i$'s is indeed an embedded $J$-holomorphic sphere of self-intersection $(-1)$ in $X'$.

\begin{prop}
\label{p:grcomp1}
Let $\Sigma_{\epsilon}^{-1}$ and $v^i$, $i \in \left\{1, \ldots, N \right\}$, as in the above paragraph. Then for at least one $i$, the simple $J$-holomorphic curve $v^i$ is an embedded sphere with self-intersection $(-1)$.
\end{prop}

\begin{proof}
If $N=1$, then $m_1=1$ and $c_1(X') \cdot [v^1] = 1$, applying the inequality (\ref{eq:jhadj}) for $v^1$, we have that $[v^1]^2 \geq -1$. If $[v^1]^2 = -1$, then by Corollary~\ref{cor:adjformula} it must be an embedding. If $[v^1]^2 = k \geq 0$, then by Lemma~\ref{l:ss}, there exists an embedded symplectic surface $v^1_S$, with $[u^1] = [v^1_S]$, for which we have $-\chi(v^1_S) = [v^1_S]^2 - c_1(X') \cdot [v^1_S] = k-1$. However, if this is the case then this violates the generalized adjunction formula, since we would then have $k-1 \geq k + |c_1(X') \cdot [v^1_S]|$, which cannot occur. 

We will prove this proposition for general $N$ with an inductive combinatorial argument using Corollary~\ref{cor:adjformula}, Lemma~\ref{l:ss}, the adjunction formula for embedded symplectic surfaces, as well as the generalized adjunction formula (Theorem~\ref{thm:genadjfor}). First, (although not strictly necessary for the proof), we will prove the proposition for $N=2$, and make a slightly stronger assumption for the initial inductive case, in order to go to the general inductive step in a less cumbersome manner. If $N=2$, then we have:
\begin{equation}
\label{eq:m1m2}
[\Sigma_{\epsilon}^{-1}] = m_1[v^1] + m_2[v^2] \,\,\,\,\,\,\,\,\,\,\,\,\,\,\, m_1 c_1(X') \cdot [v^1] + m_2 c_1(X') \cdot [v^2] = 1
\end{equation}

\textit{Case 1:} Assume $[v^1]^2 = 2k  \geq 0$, then by inequality (\ref{eq:jhadj}), we have: $c_1(X') \cdot [v^1] \leq 2 + 2k$, therefore:
\begin{eqnarray*}
\text{If} \, c_1(X') \cdot [v^1] = 2 + 2k &\Rightarrow& v^1 \,\,\text{must be embedded}   \\
\text{If} \, c_1(X') \cdot [v^1] = 2k &\Rightarrow& \exists v^1_S \,\,\text{s.t.}\,\, -\chi(v^1_S) = 0   \\
\text{If} \, c_1(X') \cdot [v^1] = 2k -2 &\Rightarrow& \exists v^1_S \,\,\text{s.t.}\,\, -\chi(v^1_S) = 2   \\
\,\,\,\,\,\,\,\,\vdots \,\,\,\,\,\,\,\, &\Rightarrow& \,\,\,\,\,\,\,\, \vdots \,\,\,\,\,\,\,\,   \\
\text{If} \, c_1(X') \cdot [v^1] = 2 &\Rightarrow& \exists v^1_S \,\,\text{s.t.}\,\, -\chi(v^1_S) = 2k-2  
\end{eqnarray*}

\noindent where $v^1_S$ is an embedded symplectic surface such that $[v^1]=[v^1_S]$. This forces $c_1(X') \cdot [v^1] \leq 0$, since if $ 2 \leq c_1(X') \cdot [v^1] \leq 2+2k$, then the embedded surface $v^1_S$ fails to satisfy the generalized adjunction formula (Theorem~\ref{thm:genadjfor}). Also, note that $c_1(X') \cdot [v^1]$ must be an even integer. Next, (\ref{eq:m1m2}) together with $c_1(X') \cdot [v^1] \leq 0$ imply that $c_1(X') \cdot [v^2] \geq 1$. If we apply (\ref{eq:jhadj}) to $v^2$, we get: $[v^2]^2 \geq -1$. Thus, if $[v^2]^2 = -1$, then $c_1(X') \cdot [v^2] = 1$ and by Corollary~\ref{cor:adjformula} $v^2$ is an embedding, if not then $[v^2]^2 = l \geq 0$ and by (\ref{eq:jhadj}) we get $1 \leq c_1(X') \cdot [v^2] \leq l+2$, so:
\begin{eqnarray*}
\text{If} \, c_1(X') \cdot [v^2] = l + 2 &\Rightarrow& \exists v^2_S \,\,\text{s.t.}\,\, -\chi(v^2_S)=-2   \\
\text{If} \, c_1(X') \cdot [v^2] = l &\Rightarrow& \exists v^2_S \,\,\text{s.t.}\,\, -\chi(v^2_S)=0   \\
\text{If} \, c_1(X') \cdot [v^2] = l - 2 &\Rightarrow& \exists v^2_S \,\,\text{s.t.}\,\, -\chi(v^2_S)=2  \\
\,\,\,\,\,\,\,\,\vdots \,\,\,\,\,\,\,\, &\Rightarrow& \,\,\,\,\,\,\,\, \vdots \,\,\,\,\,\,\,\,   \\
\text{If} \, c_1(X') \cdot [v^2] = 1 &\Rightarrow& \exists v^2_S \,\,\text{s.t.}\,\, -\chi(v^2_S)=l-1 (\text{if}\,l\,\text{is even})  
\end{eqnarray*}

\noindent where $v^2_S$ is an embedded symplectic surface such that $[v^2]=[v^2_S]$. Here, we must have $[v^2]^2 = -1$, since all the cases where $[v^2]^2 = l \geq 0$ and $1 \leq c_1(X') \cdot [v^2] \leq l+2$, cannot occur because applying the generalized adjunction formula (Theorem~\ref{thm:genadjfor}) would result in a contradiction. Consequently, if $[v^1]^2 = 2k \geq 0$, then we must have $[v^2]^2= -1$ and $v^2$ must be an embedded sphere. 

\textit{Case 2:} Assume $[v^1]^2 = 2k+1  \geq 0$. If we apply the inequality (\ref{eq:jhadj}) to $v^1$, then we have $c_1(X') \cdot [v^1] \leq 3+2k$. However, just as in \textit{Case 2}, if $1 \leq c_1(X') \cdot [v^1] \leq 3+2k$, then there would exist an embedded symplectic surface $v^1_S$, with $[v^1]=[v^1_S]$, such that applying the generalized adjunction formula (Theorem~\ref{thm:genadjfor}) for $v^1_S$ would result in a contradiction. Also, as before, we again observe that the integer $[v^1]^2 - c_1(X') \cdot [v^1]$ must be even, thus we have $c_1(X') \cdot [v^1] \leq -1$.

We proceed as before in \textit{Case 1}, and $c_1(X') \cdot [v^1] \leq -1$ together with equation (\ref{eq:m1m2}), imply that $1 \leq c_1(X') \cdot [v^2]$. Therefore, by the same steps as in \textit{Case 1}, if $[v^1]^2 = 2k+1$, then we must have $[v^2]^2 = -1$, and $v^2$ must be an embedded sphere.

We can switch the roles of $v^1$ and $v^2$, in the above cases, which implies that if $[v^2]^2 = k \geq 0$ then $[v^1]^2 = -1$ and $v^1$ must be an embedded sphere. Therefore, we are left with case:

\textit{Case 3:} Assume both $[v^1]^2 \leq -1$ and $[v^2]^2 \leq -1$. We can again apply inequalities (\ref{eq:jhadj}) to $v^1$ and $v^2$, multiplying the first by $m_1$ and the second one by $m_2$,  adding them together, and using (\ref{eq:m1m2}), we get:
\begin{equation*}
\label{eq:m1m2ineq}
1 - 2m_1 -2m_2 \leq m_1[v^1]^2+m_2[v^2]^2 \, ,
\end{equation*}

\noindent implying that both $[v^1]^2$ and $[v^2]^2$ can't be $\leq 2$. Therefore, we are left with a finite number of possibilities: Either $[v^1]^2 = -1$ and $[v^2]^2 = -k \leq -1$ (satisfying inequality (\ref{eq:m1m2ineq})) or the same with roles of $v^1$ and $v^2$ switched. In this case we have the following:
 
\[
\left.\begin{array}{l}
    [v^1]^2 = -1 \\ \left[v^2\right]^2 = -k \leq -1 
\end{array}\right\} 
\implies 
\left.\begin{array}{l}
    c_1(X') \cdot [v^1] \leq 1  \\ c_1(X') \cdot [v^2] \leq 2-k .
\end{array}\right.
\]

\noindent If $k=1$ then, (\ref{eq:m1m2}) implies that at least one of $c_1(X') \cdot [v^i]$ must be $1$, in turn implying that either $v^1$ or $v^2$ is an embedded sphere with self-intersection $(-1)$. If $k > 1$, then again because of $(\ref{eq:m1m2})$, we must have $c_1(X') \cdot [v^1]=1$ implying that $v^1$ is an embedded sphere with self-intersection $(-1)$. If roles of $v^1$ and $v^2$ are switched, with $[v^1]^2 = -k \leq -2$ and $[v^2]^2 = -1$, we would have $v^2$ be an embedded sphere with self-intersection $(-1)$.

This covers all the possibilities of the values for $[v^1]^2$ and $[v^2]^2$ with $N=2$, and in each case at least one of $v^1$ or $v^2$ is an embedded sphere with self-intersection $(-1)$. We observe that if we replace the heavily used equation (\ref{eq:m1m2}), by:
\begin{equation}
\label{eq:m1m2m}
m_1 c_1(X') \cdot [v^1] + m_2 c_1(X') \cdot [v^2] = m \geq 1
\end{equation}

\noindent then everything in the \textit{Cases 1-3} would proceed in the same way. In \textit{Case 1}, whenever we have $c_1(X') \cdot [v^1] \leq 0$, we can still use equation (\ref{eq:m1m2m}) to conclude that $c_1(X') \cdot [v^2] \geq 1$, and everything would proceed in the same way. In \textit{Case 2}, whenever we have $c_1(X') \cdot [v^1] \leq -1$, again we can still use equation (\ref{eq:m1m2m}) to conclude that $c_1(X') \cdot [v^2] \geq 1$. Likewise in \textit{Case 3}, $m \geq 1$ in (\ref{eq:m1m2m}) is all that is needed to reach the desired conclusion. 

Consequently, for a configuration of $J$-holomorphic curves $m_1[v^1]+m_2[v^2]$, with the condition (\ref{eq:m1m2m}), at least one of the curves $v^1$ and $v^2$ must be an embedded sphere with self-intersection $(-1)$. We make an induction assumption, that if we have a configuration of $J$-holomorphic curves $m_1[v^1]+m_2[v^2]+ \cdots + m_{N-1}[v^{N-1}]$, with the condition:
\begin{equation}
\label{eq:m1mn1m}
m_1 c_1(X') \cdot [v^1] + m_2 c_1(X') \cdot [v^2] + \cdots + m_{N-1} c_1(X') \cdot [v^{N-1}] = m \geq 1
\end{equation}

\noindent then one of the $v^i$s, $1 \leq i \leq N-1$, is an embedded sphere with self-intersection $(-1)$. We will show that if we have a configuration of $J$-holomorphic curves $m_1[v^1]+m_2[v^2]+ \cdots + m_{N}[v^{N}]$, with the condition:
\begin{equation}
\label{eq:m1mnm}
m_1 c_1(X') \cdot [v^1] + m_2 c_1(X') \cdot [v^2] + \cdots + m_{N} c_1(X') \cdot [v^{N}] = m' \geq 1
\end{equation}

\noindent then one of the $v^i$s, $1 \leq i \leq N$, is an embedded sphere with self-intersection $(-1)$.

\textit{Case 1':} Assume $[v^N]^2 =2k \geq 0$. Then by (\ref{eq:jhadj}), we have $c_1(X') \cdot [v^N] \leq 2k+2$. However, as in \textit{Case 1} from $N=2$, by Lemma~\ref{l:ss} the existence of a smooth symplectic surface $v^N_S$, with $[v^N]=[v^N_S]$, together with the generalized adjunction formula (Theorem~\ref{thm:genadjfor}) imply that in fact $c_1(X') \cdot [v^N] \leq 0$. Combining this with (\ref{eq:m1mnm}), we get:
\begin{equation}
\begin{split}
m' - m_1 c_1(X') \cdot [v^1] - \cdots - m_{N-1} c_1(X') \cdot [v^{N-1}] = m_N c_1(X') \cdot [v^N] \leq 0 \nonumber
\end{split}
\end{equation}
\vspace{-.2in}
\begin{equation}
\Rightarrow \,\,\,\, m_1 c_1(X') \cdot [v^1] + \cdots + m_{N-1} c_1(X') \cdot [v^{N-1}] \geq m' \geq 1 \nonumber 
\end{equation}

\noindent which according to the induction hypothesis implies that at least one $v^i$s, $1 \leq i \leq N-1$, is an embedded sphere with self-intersection $(-1)$.

\textit{Case 2':} Assume $[v^N]^2 =2k+1 \geq 1$. Again, by (\ref{eq:jhadj}), we have $c_1(X') \cdot [v^N] \leq 2k+3$. However, as in \textit{Case 2} from $N=2$, we have $c_1(X') \cdot [v^N] \leq -1$, and combining this with (\ref{eq:m1mnm}), we get:
\begin{equation}
\begin{split}
& m' - m_1 c_1(X') \cdot [v^1] - \cdots - m_{N-1} c_1(X') \cdot [v^{N-1}] \\
& = m_N c_1(X') \cdot [v^N] \leq -m_N \nonumber
\end{split}
\end{equation}
 
\begin{equation}
\Rightarrow \,\,\, m_1 c_1(X') \cdot [v^1] + \cdots + m_{N-1} c_1(X') \cdot [v^{N-1}] \geq m' + m_N \geq 1  \nonumber 
\end{equation}


\noindent which again according to the induction hypothesis implies that at least one $v^i$s, $1 \leq i \leq N-1$, is an embedded sphere with self-intersection $(-1)$.

\textit{Case 3':} Assume $[v_N]^2=-1$. Applying (\ref{eq:jhadj}) to $v^N$, we get $c_1(X') \cdot v_N \leq 1$. If $c_1(X') \cdot v_N = 1$, then $v_N$ is an embedded sphere. Otherwise, $c_1(X') \cdot v_N \leq -1$, and as in \textit{Case 2'}, we have:
\begin{equation}
m_1 c_1(X') \cdot [v^1] + \cdots + m_{N-1} c_1(X') \cdot [v^{N-1}] \geq m' + m_N \geq 1
\end{equation}

\noindent which by the induction hypothesis would imply that at least one of the $v^i$s, for $1 \leq i \leq N-1$, is an embedded sphere with self-intersection $(-1)$.

\textit{Case 4'} Assume $[v_N]^2=-2$. Again, applying (\ref{eq:jhadj}) to $v^N$, we get $c_1(X') \cdot v_N \leq 0$, meaning we have:
\begin{equation}
m_1 c_1(X') \cdot [v^1] + \cdots + m_{N-1} c_1(X') \cdot [v^{N-1}] \geq m' \geq 1
\end{equation}

\noindent which by the induction hypothesis again would imply that at least one of the $v^i$s, for $1 \leq i \leq N-1$, is an embedded sphere with self-intersection $(-1)$.

Applying \textit{Cases 1'-4'} to every $v^i$, $1 \leq i \leq N-1$, and applying the induction hypothesis each time, gives us that for all the instances where $[v^i]^2 \geq -2$, for any $1 \leq i \leq N$, we will have a $v^j$, for at least one $1 \leq j \leq N$, that is an embedded sphere with self-intersection $(-1)$. Therefore, the only remaining cases is when $[v^i]^2 = -k_i \leq -3$ for all $1 \leq i \leq N$. In which case, we would have $c_1(X') \cdot [v^i] \leq 2-k_i \leq -1$ for all $1 \leq i \leq N$, which would violate the assumption (\ref{eq:m1mnm}). This concludes the induction argument. As a result, when $m'=1$, this is the case of the Proposition~\ref{p:grcomp1}. \end{proof}

As a result of Propostion~\ref{p:grcomp1}, we now have a $J$-holomorphic embedded sphere of self-intersection $(-1)$ in $(X',\omega')$, which we will name $\Sigma_{-1}$, along with a $C_n$ configuration of $J$-holomorphic spheres. This proves Proposition~\ref{p:jholosphere} \end{proof}

An important feature of $J$-holomorphic curves, proven by McDuff \cite{McS3}, is that their intersections are always positive. In fact, we can always perturb a set of $J$-holomorphic curves and obtain embedded symplectic surfaces intersecting positively and transversally. Li-Usher in \cite{LiUs}, develop McDuff's techniques further, in order to perturb several $J$-holomorphic curves at once, and obtain the following result:

\begin{lemma}
\label{l:lius}
\cite{LiUs} Any set of distinct $J$-holomorphic curves $C_0, \ldots , C_m$ can be perturbed to symplectic surfaces $C_0', \ldots , C_m'$ whose intersections are all transverse and positive, with $C_i' \cap C_j' \cap C_k' = \varnothing$ when $i,j,k$ are all distinct. Furthermore, there is an almost-complex structure $J'$ arbitrarily $C^1$-close to $J$ such that the $C_i'$ are $J'$-holomorphic.
\end{lemma}

\noindent This is shown by modeling a neighborhood around each intersection point or singularity with holomorphic coordinates, and then slightly perturbing each branch.

Proposition~\ref{p:sympemb} is now a direct consequence of Lemma~\ref{l:lius}, Lemma~\ref{l:jholocn} and Proposition~\ref{p:jholosphere}.


\subsection{Step 2}
\label{sec:step2}
In this next step of our proof of Theorem~\ref{thm:sympemb}, we use $\Sigma_{-1}$ to construct a homology class $\gamma$ and compute $c_1(X) \cdot \gamma$ in terms of the intersection numbers of $\Sigma_{-1}$ with the spheres of the $C_n$ configuration.

We begin by rationally blowing down the $C_n$ configuration in $(X',\omega')$ symplectically. We can do so by the definition of the symplectic rational blow-down of Symington in \cite{Sym3}. We choose a neighborhood $(N(C_n),\omega'|_{N(C_n)})$ of the spheres in $C_n$, such that $\partial(N(C_n)) \cap \Sigma_{-1} \cong S^1$, $N(C_n) \cap \Sigma_{-1} \cong D^2$ and $(X' \backslash N(C_n)) \cap \Sigma_{-1} \cong D^2$. We denote this rational blow-down of $(X',\omega')$ as $(\tilde{X}',\tilde{\omega}')$. We observe that the symplectic manifolds $(X,\omega)$ and $(\tilde{X}',\tilde{\omega}')$ differ only by the volume of the rational homology ball $B_n$. This is due to the non-uniqueness of the symplectic rational blow-up operation, in terms of the symplectic volume of the $B_n$s. This also implies that the symplectic rational blow-down and the symplectic rational blow-up are not strictly inverse operations. However, $(\tilde{X}',\tilde{\omega}')$ still has the properties that $(X,\omega)$ does: $[c_1(\tilde{X}',\tilde{\omega}')] = -\left[\tilde{\omega}' \right]$, $b_2^+(\tilde{X}') > 1$, $\mathcal{B}as_X = \left\{\pm c_1(\tilde{X}',\tilde{\omega}')\right\}$ and $n \geq c_1^2(\tilde{X}',\tilde{\omega}') + 2$. Therefore, for the remainder of the proof, we will abuse notation and write $(X,\omega)$ for $(\tilde{X}',\tilde{\omega}')$.

Back up in $X'$, we can split up rational homology classes as follows:
 
$$
\begin{array}{rclll}
H_2(X';\Q) & = & H_2(C_n;\Q) & \oplus & H_2(X' \backslash C_n;\Q) \\
\Sigma_{-1} & = & \,\,\,\,\,\,\,\,\,\,a & + & \,\,\,\,\,\,\,\,\,\,b \\
PD(c_1(X',\omega')) & = & \,\,\,\,\,\,\,\,\,\,c & + & \,\,\,\,\,\,\,\,\,\,d \\
\end{array}
$$

\noindent Since we have $c_1(X',\omega') \cdot [\Sigma_{-1}] = 1$, then we have $1=a \cdot c + b \cdot d$. 

Let $D$ be a $2$-disk defined by:
\begin{equation}
D = (X'\backslash N(C_n)) \cap \Sigma_{-1} \subset X.
\end{equation}

\noindent Observe that $D \subset X$, since by definition $X \cong (X' \backslash N(C_n)) \cup B_n$. Also, since $\partial D \subset \partial B_n$ and $H_1(B_n;\Z) \cong \Z/n\Z$, then $n \partial D \cong 0 \in  H_1(B_n;\Q)$. Back down in $X$, we can now define the class $\gamma \in H_2(X;\Q)$ by:
\begin{equation}
\gamma = nD + e^2
\end{equation}

\noindent where $e^2$ is just a $2$-cell in $B_n \subset X$ for which $\partial (e^2) = n \partial D$. Since, $c_1(X',\omega') \cdot [\Sigma_{-1}] = a \cdot c + b \cdot d$ and $H_2(B_n;\Q)$ is trivial, we have:
\begin{equation}
c_1(X,\omega) \cdot \gamma = n b \cdot d \, .
\end{equation}

Our goal is to compute $c_1(X,\omega) \cdot \gamma $ explicitly in terms of the intersections of the sphere $\Sigma_{-1}$ with the spheres of $C_n$. Next, in \textit{\textbf{Step 3}} we will show that whenever $c_1(X,\omega) \cdot \gamma > 0$ then we also have $\omega \cdot \gamma > 0$, thus contradicting the condition $[c_1(X,\omega)] = -\left[\omega \right]$. In \textit{\textbf{Step 4}} we will show that the intersection configurations of $\Sigma_{-1}$ with $C_n$ yielding $c_1(X,\omega) \cdot \gamma \leq 0$ will also produce a contradiction.

In order to compute $c_1(X,\omega) \cdot \gamma $, all we need to compute is $a \cdot c$, since $n b \cdot d = n(1 - a \cdot c)$, which is fairly standard. Recall, we denote the spheres of the $C_n$ configuration by $S_1,S_2,S_3, \ldots, S_{n-1}$, with $[S_1]^2 = -n-2$ and $[S_i]^2 = -2$ for $2 \leq i \leq n-1$. Thus, we may denote the basis of $H_2(C_n;\Q)$ by $[S_1], [S_2], [S_3], \ldots, [S_{n-1}]$. As a result ``$a$", the homology class of $\Sigma_{-1}$ lying in $H_2(C_n;\Q)$, may be expressed as:
\begin{equation}
a = a_1[S_1] + a_2[S_2] + a_3[S_3] + \cdots + a_{n-1}[S_{n-1}]
\end{equation}

\noindent where $a_i \in \Q$. Next, let $I_j$ be the intersection numbers of $[\Sigma_{-1}]$ and $ [S_j]$:
\begin{eqnarray*}
[\Sigma_{-1}] \cdot [S_1] & = & I_1   \\
\left[\Sigma_{-1}\right] \cdot [S_2] & = & I_2   \\
\left[\Sigma_{-1}\right] \cdot [S_3]& = & I_3   \\
\vdots & = & \vdots   \\
\left[\Sigma_{-1}\right] \cdot [S_{n-1}]& = & I_{n-1} \, .
\end{eqnarray*}

\noindent (Note, we have $\alpha_j = I_j$ (see Definition~\ref{d:bntype}), since the intersections of the sphere $\Sigma_{-1}$ with the spheres $S_j$ are positive and transverse.) In order to express the $a_i$ in terms of the intersection numbers $I_j$, we need to solve the following linear system:
\begin{eqnarray*}
(a_1[S_1] + a_2[S_2] + a_3[S_3] + \cdots + a_{n-1}[S_{n-1}]) \cdot [S_1] & = & I_1   \\
(a_1[S_1] + a_2[S_2] + a_3[S_3] + \cdots + a_{n-1}[S_{n-1}]) \cdot [S_2] & = & I_2  \\
(a_1[S_1] + a_2[S_2] + a_3[S_3] + \cdots + a_{n-1}[S_{n-1}]) \cdot [S_3] & = & I_3   \\
\vdots & = & \vdots   \\
(a_1[S_1] + a_2[S_2] + a_3[S_3] + \cdots + a_{n-1}[S_{n-1}]) \cdot [S_{n-1}]& = & I_{n-1} \, .
\end{eqnarray*}

Next, we can express ``$c$", the homology class of $PD(c_1(X',\omega'))$ lying in $H_2(C_n;\Q)$, in terms of the basis $[S_1],[S_2],[S_3],\ldots,[S_{n-1}]$:
\begin{equation}
c = c_1[S_1] + c_2[S_2] + c_3[S_3] + \cdots + c_{n-1}[S_{n-1}]
\end{equation}

\noindent where the $c_i \in \Q$. Since the $S_i$ are symplectic spheres, we have the following:
\begin{eqnarray*}
c_1(X') \cdot [S_1] & = & -n   \\
c_1(X') \cdot [S_2] & = & 0   \\
c_1(X') \cdot [S_3]& = & 0   \\
\vdots & = & \vdots   \\
c_1(X') \cdot [S_{n-1}]& = & 0 \, .
\end{eqnarray*}

As a result, the quantity $a \cdot c$ is the dot product of the following two vectors in $H_2(C_n;\Q)$:
\begin{equation}
[a_1, a_2, a_3, \ldots, a_{n-1}]
\end{equation}

\noindent and
\begin{equation}
[-n, 0, 0, \ldots, 0] \, .
\end{equation}

\noindent Consequently, we only have to compute $a_1$ in terms of the intersection numbers $I_j$, which corresponds to the first row of the inverse of the $H_2(C_n;\Z)$ intersection matrix, giving us:
\begin{equation}
a_1 = \frac{-n+1}{n^2}I_1 + \frac{-n+2}{n^2}I_2 + \cdots + \frac{-2}{n^2}I_{n-2} + \frac{-1}{n^2}I_{n-1} \, .
\end{equation}

\noindent Since $a \cdot c = a_1 \cdot n$ and $c_1(X,\omega) \cdot \gamma = n(1- a \cdot c)$, we finally get:
\begin{equation}
\label{eq:c1gamma}
c_1(X,\omega) \cdot \gamma = n - I_{n-1} -2I_{n-2} -3I_{n-3} - \cdots -(n-2)I_2 - (n-1)I_1 \, .
\end{equation}

Note, that since $\alpha_j = I_j$, then we have shown that if the symplectic embedding of $B_n \hookrightarrow X$ is of type $\langle \alpha_1, \alpha_2, \alpha_3, \ldots , \alpha_{n-1} \rangle$, then there is a class $\gamma$, such that $c_1(X) \cdot \gamma$ is given by (\ref{eq:c1gamma}).

\vspace{.2in}

\subsection{Step 3} 
\label{sec:step3}
In this step we will show that if $c_1(X,\omega) \cdot \gamma > 0$, then we must also have $\omega \cdot \gamma > 0$, thus violating the $[c_1(X,\omega)] = -[\omega]$ condition of $(X,\omega)$. This will eliminate the possibility of embeddings $B_n \hookrightarrow X$  of type $\A_1 \subset \A$, where $\A_1$ is the set of $(n-1)-$tuples $\langle \alpha_1, \alpha_2, \alpha_3, \ldots , \alpha_{n-1} \rangle$ satisfying the inequality (\ref{eq:c1g0}) (with $\alpha_j = I_j$).
 
If $c_1(X) \cdot \gamma > 0$, we have:
\begin{equation}
\label{eq:c1g0}
n - I_{n-1} -2I_{n-2} -3I_{n-3} - \cdots -(n-2)I_2 - (n-1)I_1 > 0 .
\end{equation}

\noindent First, we will use the following lemma to rule out some cases.

\begin{lemma}
\label{l:int12}
Let $\Sigma$ and $S$ be embedded spheres in a smooth $4$-manifold $M$ with $b_2^+(M) > 1$, such that $[\Sigma]^2 = -1$ and $[S]^2 = -2$. Assume $[\Sigma] \cdot [S] = k \geq 1$, then we must have $k=1$.
\end{lemma} 

\begin{proof}
The proof will follow from the following proposition:

\begin{prop}
\label{p:reflection}
\cite{FrMo} Let $M$ be an oriented $4$-manifold and $S^2 \subset M$ be an embedded sphere with $\alpha \in H^2(M;\Z)$ the cohomology class dual to $S^2$. If $\alpha^2 = -1$ or $-2$, there is an orientation preserving self-diffeomorphism $\varphi$ of $M$ such that $\varphi^{\ast} = R_{\alpha}$, where:
\begin{equation}
R_{\alpha}(x) = x + 2(x \cdot \alpha)\alpha
\end{equation}

\noindent if $\alpha^2 = -1$ and
\begin{equation}
R_{\alpha}(x) = x + (x \cdot \alpha)\alpha
\end{equation}

\noindent if $\alpha^2 = -2$. (Note, in both cases $R_{\alpha} = -\alpha$ and $R_{\alpha}^2 =$ Id, hence often referred to as the reflection automorphism.)

\end{prop}

As a result of this proposition, the spheres $\Sigma$ and $S$ will induce orientation preserving diffeomorphisms on $M$, corresponding to the following reflection automorphisms on $H_2(M;\Z)$:
\begin{eqnarray}
R_{\Sigma}(x) &=& x + 2(x \cdot [\Sigma])[\Sigma] \\
R_{S}(x) &=& x + (x \cdot [S])[S] \, .
\end{eqnarray}

\noindent We begin with applying $R_{S}$ to $x=[\Sigma]$:
\begin{equation}
R_{S}([\Sigma]) = [\Sigma] + k[S] \, .
\end{equation}

\noindent Next, we apply $R_{\Sigma}$ to $x=[\Sigma] + k[S]$:
\begin{equation}
R_{\Sigma}([\Sigma] + k[S]) = (2k^2-1)[\Sigma] + k[S] \, .
\end{equation}

\noindent In this manner, we can continue to alternately apply $R_{S}$ and $R_{\Sigma}$, and get:
\begin{eqnarray*}
R_{S}((2k^2-1)[\Sigma] + k[S]) &=& (2k^2-1)[\Sigma] + (2k^3 - 2k)[S] : = [A_S]  \\
R_{\Sigma}([A_S]) &=& (4k^4-6k+1)[\Sigma] + (2k^3 - 2k)[S] := [A_{\Sigma S}]  \\
R_{S}([A_{\Sigma S}]) &=& (4k^4-6k+1)[\Sigma] + (4k^5-8k^3+3k)[S]   \\
\vdots &=& \vdots   \, .
\end{eqnarray*} 

\noindent We observe that as long as $k \geq 2$, the polynomials above keep growing, thus implying that there is an infinite number of spheres with homology classes of the form $x=s_1[\Sigma] + s_2[S]$ with $x^2=-1$. This cannot occur, since if it did, it would imply that there is an infinite number of Seiberg-Witten basic classes of the manifold $M$, which cannot happen if $b_2^+(M) > 1$. \end{proof}

Lemma~\ref{l:int12} immediately implies the following Corollary:   

\begin{cor}
\label{cor:I2In1}
With the intersection numbers $I_j = [\Sigma_{-1}] \cdot [S_j]$, as in section~\ref{sec:step2}, we must have $I_2 + I_3 + I_4 + \cdots + I_{n-1} \leq 1$. 
\end{cor} 

\begin{proof}
The spheres $S_j$ with $2 \leq j \leq n-1$ intersect transversally with the neighboring spheres in the plumbing configuration $C_n$. Therefore, we can construct the sphere $S_2^{n-1}$, which is the union of the spheres $S_j, 2 \leq j \leq n-1$, with all the transverse intersection points smoothed out. The sphere $S_2^{n-1}$ has self-intersection $(-2)$, since its homology class is:
\begin{equation}
[S_2^{n-1}] = [S_2] + [S_3] + [S_4] + \cdots + [S_{n-1}] \, .
\end{equation}

\noindent Now we can apply Lemma~\ref{l:int12} with $\Sigma = \Sigma_{-1}$ and $S = S_2^{n-1}$, and conclude that $[\Sigma_{-1}] \cdot [S_2^{n-1}]$ is at most $1$, implying:
\begin{equation}
[\Sigma_{-1}] \cdot [S_2^{n-1}] = I_2 + I_3 + I_4 + \cdots + I_{n-1} \leq 1 \, .
\end{equation}   \end{proof}                                                                                

As a direct consequence of Corollary~\ref{cor:I2In1} and (\ref{eq:c1gamma}), we have the following:
 
\begin{cor}
\label{cor:c1pos}
If $c_1(X,\omega) \cdot \gamma > 0$, with $\gamma = nD + e^2$ as defined in section~\ref{sec:step2}, then there is only one $j$, $1 \leq j \leq n-1$, for which $I_j = 1$ and $I_k = 0$ if $j \neq k$.
\end{cor}

Next, we will use toric and almost-toric fibrations, introduced in section~\ref{sec:toric}, to show that for those cases where $c_1(X,\omega) \cdot \gamma > 0$, we have $\omega \cdot \gamma > 0$.

\begin{prop}
\label{p:omegapos}
If there is only one $j$, $1 \leq j \leq n-1$, for which $I_j = 1$ and $I_k = 0$ if $j \neq k$, then $\omega \cdot \gamma > 0$.
\end{prop}

\begin{figure}[ht!]
\labellist
\small\hair 2pt
\pinlabel $S_1$ at 25 0
\pinlabel $S_2$ at 58 4
\pinlabel $S_3$ at 80 8
\pinlabel $S_{n-1}$ at 103 21
\pinlabel $L(n^2,n-1)$ at 48 65
\pinlabel $\mu^1_1$ at 30 35
\pinlabel $S_1$ at 142 0
\pinlabel $S_2$ at 175 4
\pinlabel $S_3$ at 197 8
\pinlabel $S_{n-1}$ at 220 21
\pinlabel $L(n^2,n-1)$ at 165 65
\pinlabel $\mu^1_2$ at 155 35
\pinlabel $S_1$ at 287 0
\pinlabel $S_2$ at 320 4
\pinlabel $S_3$ at 342 8
\pinlabel $S_{n-1}$ at 365 21
\pinlabel $L(n^2,n-1)$ at 310 65
\pinlabel $\mu^1_{n-1}$ at 300 30
\endlabellist
\centering
\includegraphics[height = 35mm, width = 130mm]{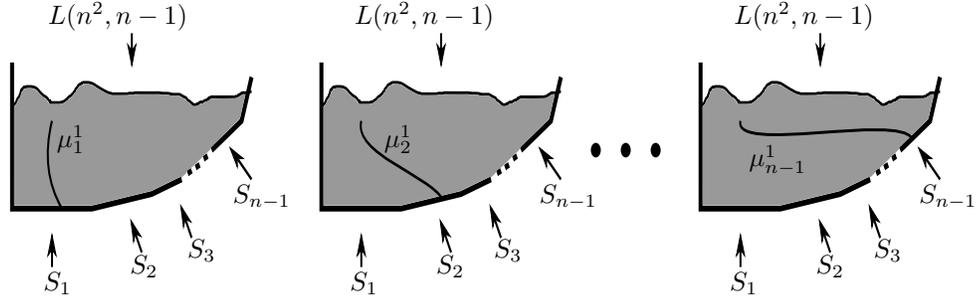}
\caption{Visible surfaces represented by curves $\mu_j^1$ in a toric model for $C_n$}
\label{f:toriccnD}
\end{figure}

\begin{proof}
If there is only one $j$, $1 \leq j \leq n-1$, for which $I_j = 1$ and $I_k = 0$ if $j \neq k$, then by definition, the sphere $\Sigma_{-1}$ only intersects the sphere $S_j$ of the $C_n$ configuration once at a point $a_j$. We can present part of $\Sigma_{-1}$ as it intersects $S_j$, by a \textit{visible surface} (see Definition~\ref{def:vissurf}), with the curve $\mu_j^1$ and a compatible covector $u_j = \left[\begin{array}{c}-1\\n+1\end{array}\right]$, for all $1 \leq j \leq n-1$, (see Figure~\ref{f:toriccnD}). We have $[\Sigma_{-1}] \cdot [S_j] = 1$, since $|u_j \times v_j | = 1$ for all $j$, where the $v_j = \left[\begin{array}{c}1-j\\(j-1)n +j\end{array}\right]$ are the \textit{collapsing covectors} corresponding to the part of the $1$-stratum that represents the spheres $S_j$, thus satisfying item (2) in the definition of visible surfaces (Definition~\ref{def:vissurf}). 

After we perform the rational blow-down, as we do in the beginning of \textit{Step 2}, we obtain the almost-toric base, as seen in Figure~\ref{f:toricbnD} (also see Figure~\ref{f:toricrbd}). We recall here that the class $\gamma = nD + e^2$, where $D$ is the ``remains" of $\Sigma_{-1}$ in $X$: $D = (X'/N(C_n)) \cap \Sigma_{-1}$. Since $\Sigma_{-1}$ is a symplectic sphere, we have $\omega \cdot nD > 0 $. In order to show that $\omega \cdot \gamma > 0$, we need to show that $\omega$ is positive on the $2$-cell, $e^2$, which ``closes up" $nD$ i.e. $\partial e^2 = \partial nD$. We will do this by exhibiting the disk $e^2 \in B_n$ as a visible surface in the almost-toric fibration of $B_n$.

\begin{figure}[ht!]
\labellist
\small\hair 2pt
\pinlabel $L(n^2,n-1)$ at 48 100
\pinlabel $L(n^2,n-1)$ at 173 100
\pinlabel $L(n^2,n-1)$ at 348 100
\pinlabel $s$ at 32 52
\pinlabel $s$ at 157 52
\pinlabel $s$ at 332 52
\pinlabel $b_1$ at 23 78
\pinlabel $b_2$ at 170 78
\pinlabel $b_3$ at 373 72
\pinlabel $\upsilon$ at 17 47
\pinlabel $\upsilon$ at 142 47
\pinlabel $\upsilon$ at 317 47
\pinlabel $\mu_1^2$ at 28 63
\pinlabel $\mu_2^2$ at 170 63
\pinlabel $\mu_{n-1}^2$ at 335 65
\endlabellist
\centering
\includegraphics[height = 40mm, width = 130mm]{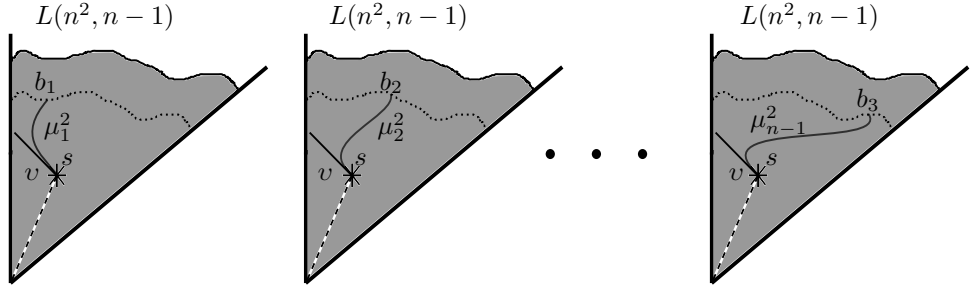}
\caption{Visible surfaces represented by curves $\mu_j^2$ in almost-toric model for $B_n$}
\label{f:toricbnD}
\end{figure}

In order to represent $e^2 \subset \gamma$ as a visible surface in the almost-toric base in Figure~\ref{f:toricbnD}, we need to choose a curve $\mu_j^2$, such that it extends the curve $\mu_j^1$ and whose collection of compatible classes in $H_2(F_b,\Z)$, for all $b \in \mu_j^2$, forms a disk. We can arrange the visible surface represented by $\mu_j^2$ to be symplectic, because of Proposition~\ref{p:vissurfarea}. Also, we can arrange $e^2$ such that it hits the Lagrangian core $\L_n$, represented by the straight line $\upsilon$, thus the curve $\mu_j^2$ hits $\upsilon$ and then the node $s$. 

On one hand, the compatible class of $\mu_j^2$ must be the same as the vanishing class of the node $s$, in order for $\mu_j^2$ to represent a visible surface (and be a disk). On the other hand, the curve $\mu_j^2$ is a continuation of the curve $\mu_j^1$. However, the curve $\mu_j^1$ represents the visible surface for $D \in C_n$. At the point $b_j$ in the toric fibration, which lies on the curve representing the boundary $\partial C_n = L(n^2,n-1)$, the compatible covector is $u_j = \left[\begin{array}{c}-1\\n+1\end{array}\right]$, corresponding to the class $\partial D \in L(n^2,n-1)$. When we begin the curve $\mu_j^2$, at the point $b_j$, the compatible class should correspond to $n \partial D \in L(n^2,n-1)$, making the compatible covector $n \left[\begin{array}{c}-1\\n+1\end{array}\right] = \left[\begin{array}{c}-n\\n^2+n\end{array}\right]$. In $\partial C_n = \partial B_n = L(n^2,n-1)$, we have the compatible class with the covector $\left[\begin{array}{c}-n\\n^2+n\end{array}\right]$ homologous to the compatible class with the covector $\left[\begin{array}{c}-1\\n\end{array}\right]$. As a result, the curve $\mu_j^2$ will have a compatible covector $u'_j = \left[\begin{array}{c}-1\\n\end{array}\right]$ for all $1 \leq j \leq n-1$, exactly the same as the vanishing covector of the node $s$. Consequently, the curves $\mu_j^2$ do indeed represent visible surfaces, the $2$-cells $e^2$ in the construction of $\gamma$.

As a result, we have explicitly exhibited that the class $\gamma = nD + e^2$ is such that $\omega \cdot \gamma > 0$, by representing $D$ and $e^2$ as visible surfaces in the almost-toric fibrations of $C_n$ and $B_n$, with positive symplectic area. \end{proof}

\begin{cor}
\label{cor:c1omega}
Let $\gamma$ be as above. If $c_1(X,\omega) \cdot \gamma  > 0$, then $\omega \cdot \gamma >0$.
\end{cor}

\begin{proof}
This is a direct consequence of Corollary~\ref{cor:c1pos} and Proposition~\ref{p:omegapos}. \end{proof}

As a result of Corollary~\ref{cor:c1omega}, we have proved that embeddings of $B_n \hookrightarrow X$ of type $\A_1 \subset \A$ cannot occur, where $\A_1$ is the set of $(n-1)$-tuples $\left\langle \alpha_1, \alpha_2, \alpha_3, \ldots, \alpha_{n-1} \right\rangle$, such that $c_1(X) \cdot \gamma  > 0$ in terms of the intersection numbers $I_j = \alpha_j$.

\vspace{.2in}

\subsection{Step 4} 
\label{sec:step4}
In this final step, we will show that symplectic embeddings of $B_n \hookrightarrow X$ of type $(\A - \A_1)$ (those of type $\A$ and not $\A_1$) and type $\E_k$, $k \geq c_1^2(X,\omega) + 2$ cannot occur. These sets of $(n-1)$-tuples precisely correspond with $c_1(X,\omega) \cdot \gamma \leq 0$, i.e. the cases where:
\begin{equation}
n - I_{n-1} -2I_{n-2} -3I_{n-3} - \cdots -(n-2)I_2 - (n-1)I_1 \leq 0 .
\end{equation}

\begin{lemma}
\label{l:I1}
Let $I_1 = [\Sigma_{-1}] \cdot [S_1]$ be as above, then $I_1 \leq n$.
\end{lemma}

\begin{proof}
First, we apply the generalized adjunction formula (Theorem~\ref{thm:genadjimm}) to the sphere $\Sigma_{-1}$, which gives us that $c_1(X',\omega') + 2[\Sigma_{-1}]$ is a SW basic class of $X'$. Second, we apply Theorem~\ref{thm:genadjimm} to the sphere $S_1$, since $[S_1]^2 = -n-2$ for a SW basic class $L$ we have: 
\begin{equation}
|L \cdot [S_1] | \leq n
\end{equation}

\noindent if we let $L = c_1(X',\omega') + 2[\Sigma_{-1}]$, then we have:
\begin{equation}
|(c_1(X',\omega') + 2[\Sigma_{-1}]) \cdot [S_1]| = |c_1(X',\omega') \cdot [S_1] + 2[\Sigma_{-1}] \cdot [S_1] | \leq n \, .
\end{equation}

\noindent Since $S_1$ is a symplectic sphere, we have $c_1(X',\omega') \cdot [S_1] = -n$, therefore, we must have: $I_1 = [\Sigma_{-1}] \cdot [S_1] \leq n$.    \end{proof}

\begin{cor}
\label{cor:I1In1}
Let $I_i = [\Sigma_{-1}] \cdot [S_i]$, $1 \leq i \leq n-1$, be as above, then $I_1 + I_2 + I_3 + \cdots + I_{n-1} \leq n$.
\end{cor}

\begin{proof}
The spheres $S_i$ intersect each other transversally in the $C_n$ configuration. We can (as done in Corollary~\ref{cor:I2In1}) construct the sphere $S^{n-1}_1$, which is the union of the spheres $S_i$, $1 \leq i \leq n-1$, with all the intersection points smoothed out (this can be done symplectically). The self-intersection number of the sphere $S^{n-1}_1$ is $(-n-2)$, since its homology class is:
\begin{equation}
[S^{n-1}_1] = [S_1] + [S_2] + [S_3] + \cdots + [S_{n-1}] .
\end{equation}

\noindent Consequently, by applying Lemma~\ref{l:I1}, since $I_1 \leq n$ then so is $I_1 + I_2 + I_3 + \cdots + I_{n-1} \leq n$. \end{proof}

In light of Corollaries~\ref{cor:I2In1}, \ref{cor:I1In1} and Lemma~\ref{l:I1}, the intersection patterns of $\Sigma_{-1}$ with the spheres of the $C_n$ configuration, giving us $c_1(X,\omega) \cdot \gamma \leq 0$, which we still have to rule out are: 

\begin{enumerate}
\item
$2 \leq I_1 \leq n$ and $I_j = 0$ for all $2 \leq j \leq n-1$

\item
$1 \leq I_1 \leq n-1$ and $I_j = 1$ for one $2 \leq j \leq n-1$ .
\end{enumerate}

\begin{lemma}
\label{l:I1n}
The following intersection configurations:
\begin{enumerate}[label=\emph{(\alph*})]
\item
$I_1 = n$ and $I_j = 0$ for all $2 \leq j \leq n-1$

\item
$I_1 = 1$ and $I_{n-1} = 1$ ($I_j = 0$ for all $2 \leq j \leq n-2$)
\end{enumerate}

\noindent will force the $4$-manifold $(X,\omega)$ to have basic classes in addition to $\pm K = \mp c_1(X,\omega)$, thus contradicting the hypothesis in Theorem~\ref{thm:sympemb}.
\end{lemma}

\begin{proof}
We begin with looking at the piece of the relative homology long exact sequence for the pair $(C_n,\partial C_n)$:
\begin{equation}
0 \rightarrow H_2(C_n;\Z) \stackrel{i}{\rightarrow} H_2(C_n,\partial C_n ; \Z) \stackrel{\partial}{\rightarrow} H_1(\partial C_n ; \Z) \rightarrow 0
\end{equation}

Let $\delta \in H_2(C_n , \partial C_n ;\Z)$ be a relative class that is a union of the disks $\Sigma_{-1} \cap N(C_n)$, where $N(C_n)$ is a neighborhood of the spheres of the $C_n$. For all $1 \leq j \leq n-1$, we have:
\begin{equation}
\delta \cdot [S_j] = \Sigma_{-1} \cdot [S_j] = I_j \, .
\end{equation}

In case $(a)$, for the class $-n\delta$ we have:
\begin{eqnarray*}
-n\delta \cdot [S_1] & = & -n^2  \\
-n\delta \cdot [S_2] & = & 0   \\
\vdots               & = & \vdots   \\
-n\delta \cdot [S_{n-1}] & = & 0 \, .  
\end{eqnarray*}

\noindent The relative class $-n\delta$ can be supported in the interior by the following homology class:
\begin{equation}
-n\delta = (n-1)[S_1] + (n-2)[S_2] + \cdots + [S_{n-1}]
\end{equation}

\noindent since,
\begin{eqnarray*}
((n-1)[S_1] + (n-2)[S_2] + \cdots + [S_{n-1}]) \cdot [S_1] & = & -n^2  \\
((n-1)[S_1] + (n-2)[S_2] + \cdots + [S_{n-1}]) \cdot [S_2] & = & 0  \\
\vdots     & = & \vdots   \\
((n-1)[S_1] + (n-2)[S_2] + \cdots + [S_{n-1}]) \cdot [S_{n-1}] & = & 0 \, . 
\end{eqnarray*}

In case $(b)$,  for the class $-n\delta$ we have:
\begin{eqnarray*}
-n\delta \cdot [S_1] & = & -n   \\
-n\delta \cdot [S_2] & = & 0   \\
\vdots               & = & \vdots   \\
-n\delta \cdot [S_{n-2}] & = & 0   \\
-n\delta \cdot [S_{n-1}] & = & -n \, . 
\end{eqnarray*}

\noindent In this case, the relative class $-n\delta$ can be supported in the interior by the following homology class:
\begin{equation}
-n\delta = [S_1] + 2[S_2] + \cdots + (n-1)[S_{n-1}]
\end{equation}

\noindent since,
\begin{eqnarray*}
([S_1] + 2[S_2] + \cdots + (n-2)[S_{n-2}] + (n-1)[S_{n-1}]) \cdot [S_1] & = & -n  \\
([S_1] + 2[S_2] + \cdots + (n-2)[S_{n-2}] + (n-1)[S_{n-1}]) \cdot [S_2] & = & 0  \\
\vdots     & = & \vdots  \\
([S_1] + 2[S_2] + \cdots + (n-2)[S_{n-2}] + (n-1)[S_{n-1}]) \cdot [S_{n-2}] & = & 0  \\
([S_1] + 2[S_2] + \cdots + (n-2)[S_{n-2}] + (n-1)[S_{n-1}]) \cdot [S_{n-1}] & = & -n \, . 
\end{eqnarray*}

In both cases $(a)$ and $(b)$ we have the relative class $-n\delta \in im(i) = ker(\partial)$, implying that $\partial (\delta) \in H_1(\partial C_n;\Z) \cong H_1(L(n^2,n-1) ; \Z) \cong \Z / n^2\Z$ is an element of order $n$. According to Theorem~\ref{thm:FSrbdsw} and \cite{Park}), this implies that the basic classes $\pm(c_1(X',\omega') +2[\Sigma_{-1}])$ of $(X',\omega')$ extend to basic classes of the rational blow-down $(X,\omega)$. Note, the classes $\pm c_1(X',\omega')$ must extend to the basic classes of $(X,\omega)$ since they are the $\pm$ the canonical class. Moreover, $c_1(X',\omega') + 2[\Sigma_{-1}]$ and $c_1(X',\omega')$ must extend to different basic classes on $(X,\omega)$, otherwise we would have $[\Sigma_{-1}] = 0$. Therefore, the $4$-manifold $X$ will have at least four basic classes, which is a contradiction.   \end{proof}

\begin{lemma}
\label{l:Ij1}
The following intersection configurations:
\begin{enumerate}[label=\emph{\roman*})]
\item
$I_1 = 1$ and $I_j=1$ for one $j$ such that $2 \leq j \leq n-2$

\item
$I_1 = k$ with $2 \leq k \leq n-1$ and $I_j=1$ for one $j$ such that $2 \leq j \leq n-1$

\end{enumerate}

cannot occur in $(X',\omega')$, since it is a symplectic $4$-manifold with $b^+_2(X') > 1$.
\end{lemma}

\begin{proof}

Assume $I_1 = k$ for $1 \leq k \leq n-1$ and let $K = c_1(X',\omega')$ be the (negative of the) canonical class of $(X',\omega')$. Since $S_1$ is a symplectic sphere with self-intersection $(-n-2)$, we have that $K \cdot [S_1] = -n$. Let $L$ be any SW basic class of $X'$, then according to the generalized adjunction formula for immersed spheres, Theorem~\ref{thm:genadjimm}, we have that
\begin{equation}
\label{eq:LS1n}
|L \cdot [S_1]| \leq n
\end{equation}

\noindent or $SW_{X'}(L + 2[S_1]) = SW_{X'}(L)$ if $L \cdot S_1 \geq 0$, ($SW_{X'}(L - 2[S_1]) = SW_{X'}(L)$ if $L \cdot [S_1] \leq 0$). We will produce a specific SW basic class $L$ which will fail to satisfy (\ref{eq:LS1n}) and for which $L \pm 2S_1$ cannot be a SW basic class since the symplectic $4$-manifold $X'$ is of simple type.

First, we observe, that by smoothing out the transverse intersections of the spheres in the $C_n$ configuration and the sphere $\Sigma_{-1}$, we have the following spheres in $X'$, each with self-intersection $(-1)$:
\begin{eqnarray}
\label{eq:1spheres}
&\Sigma_{-1}&  \nonumber \\
&\Sigma_{-1}&\hspace{-.12in}+ S_j \nonumber \\
&\Sigma_{-1}&\hspace{-.12in}+ S_j + S_{j-1} \nonumber \\
&\vdots& \nonumber \\
&\Sigma_{-1}&\hspace{-.12in}+ S_j + S_{j-1} + \cdots + S_2 \nonumber \\
&\Sigma_{-1}&\hspace{-.12in}+ S_j + S_{j-1} + \cdots + S_2 + S_{j+1} \nonumber \\
&\Sigma_{-1}&\hspace{-.12in}+ S_j + S_{j-1} + \cdots + S_2 + S_{j+1} + S_{j+2} \nonumber \\
&\vdots& \nonumber \\
&\Sigma_{-1}&\hspace{-.12in}+ S_j + S_{j-1} + \cdots + S_2 + S_{j+1} + S_{j+2} + \cdots + S_{n-1} .  
\end{eqnarray}
\vspace{.1in}

\noindent Second, using these spheres, we can construct several SW basic classes using Theorem~\ref{thm:genadjimm} as follows: We start off by letting $L=K$ and $x = \Sigma_{-1}$ as in Theorem~\ref{thm:genadjimm}, since $|K \cdot [\Sigma_{-1}]| \leq -1$ cannot happen,   $K+2[\Sigma_{-1}]$ must be a SW basic class. Note, that $K^2 = (K+2[\Sigma_{-1}])^2$, as required for $4$-manifolds of simple type. Next, we let $L=K+2[\Sigma_{-1}]$ and $x = \Sigma_{-1} + S_j$, and after applying Theorem~\ref{thm:genadjimm} again, we get that since $|(K+2[\Sigma_{-1}]) \cdot ([\Sigma_{-1}] + [S_j])| \leq -1$ cannot happen, then $(K + 2[\Sigma_{-1}]) + 2([\Sigma_{-1}] + [S_j])$ is a SW basic class. Proceeding in this manner, with all the spheres of (\ref{eq:1spheres}), we get that $K'$ is a SW basic class of $X'$, where $K'$ is:
\begin{equation}
\begin{split}
K'=K &+ 2(n-1)[\Sigma_{-1}] + 2(n-2)[S_j] + \cdots \\
&+ 2(n-j)[S_2] + 2(n-(j+1))[S_{j+1}] + \cdots + 2[S_{n-1}] .
\end{split}
\end{equation}

Next, we again apply Theorem~\ref{thm:genadjimm} with $L=K'$ and $x=S_1$, and as in (\ref{eq:LS1n}), we get:
\begin{eqnarray}
\label{eq:basicclass}
|K' \cdot [S_1]| & = & |K \cdot [S_1] + 2(n-1)[\Sigma{-1}] \cdot [S_1] + 2(n-j)[S_2] \cdot [S_1]| \nonumber \\
                 & = & |(2k+1)n - 2k -2j| \leq n .
\end{eqnarray}

\noindent If $k=1$, then (\ref{eq:basicclass}) becomes:
\begin{equation}
|2n - 2 - 2j| \leq n \, ,
\end{equation}

\noindent which for $n \geq 4$ and $2 \leq j \leq n-2$ cannot occur. Therefore, $K' + 2[S_1]$ is forced to be a SW basic class, however, this is impossible since $X'$ is of simple type and $(K' + 2[S_1])^2 \neq (K')^2$. Consequently, the configurations with intersection numbers $I_1 = 1$ and $I_j = 1$ for one $j$ for which $2 \leq j \leq n-2$ cannot occur.

If $2 \leq k \leq n-1$, then the inequality (\ref{eq:basicclass}) cannot hold if $2 \leq j \leq n-1$. Therefore, again $K' + 2[S_1]$ must be a SW basic class, but this cannot happen either since $X'$ is of simple type. Consequently, the configurations with the intersection numbers $I_1 = k$ with $2 \leq k \leq n-1$ and $I_j = 1$ for one $j$ for which $2 \leq j \leq n-1$ cannot occur.     \end{proof}

The results in section~\ref{sec:step2} as well as Lemmas~\ref{l:I1n} and \ref{l:Ij1} imply that if $n \geq c_1^2(X,\omega) + 2$, then there cannot be symplectic embeddings of $B_n \hookrightarrow (X,\omega)$ of type $\A$. The only configurations which remain are those with $I_1 = k$ where $2 \leq k \leq n-1$ and $I_j = 0$ for $j$ with $2 \leq j \leq n-1$, which correspond to symplectic embeddings of $B_n \hookrightarrow X$ of type $\E_k$ for $2 \leq k \leq n-1$. Next, we will show that symplectic embeddings $B_n \hookrightarrow (X,\omega)$ of type $\E_k$, $k \geq c_1^2(X,\omega)+2$, cannot occur.

\begin{rmk}
The key difference between symplectic embeddings of $B_n \hookrightarrow X$ of type $\A$ and $\E_k$ is that in the embeddings of type $\E_k$, the sphere $\Sigma_{-1}$ does not intersect any sphere with self-intersection $(-2)$, which as seen in Lemma~\ref{l:Ij1}, creates quite a few Seiberg-Witten basic classes leading to contradictions because of adjunction formulas. Therefore, in order to prevent embeddings of type $\E_k$, $k \geq c_1^2(X,\omega)+2$, we need $c_1^2(X,\omega)$ to be low enough to guarantee the existence of several spheres with self-intersection $(-1)$, in addition to $\Sigma_{-1}$.
\end{rmk}

The next Lemma will be instrumental in showing this last part of Theorem~\ref{thm:sympemb}.

\begin{lemma}
\label{l:crn}
Let $S_r^d \subset (M,\omega)$ be an immersed symplectic sphere with self-inter\-sec\-tion $r$ and $d$ double points, where $(M,\omega)$ is a symplectic $4$-manifold with $c_1^2(M,\omega) \leq -1$ and $b_2^+(M)> 1$. Let $C_n^r \subset (M,\omega)$ be the linear plumbing of symplectic spheres $S_d^r, S_2, S_3, \ldots, S_{n-1}$, where the $S_j$ are embedded symplectic spheres with $[S_j]^2 = -2$ for $2 \leq j \leq n-1$. Then there exists an embedded symplectic sphere $\hat{\Sigma'}_{-1} \subset (M, \omega)$ with $[\hat{\Sigma'}_{-1}]^2 = -1$, and $C_n^{'r} \subset (M,\omega)$, a linear plumbing configuration of symplectic spheres $S_r^d, S'_2, S'_3, \ldots, S'_{n-1}$ (each $S'_i$ is a perturbation of $S_i$), such that if $\hat{\Sigma'}_{-1}$ intersects any spheres in the $C_n^{'r}$ configuration it must do so positively and transversally.
\end{lemma}

\begin{proof}
The proof of this lemma mirrors the proof of Proposition~\ref{p:sympemb} in section~\ref{sec:step1}. As in the proof of Proposition~\ref{p:sympemb}, we start by putting an $\omega$-compatible almost-complex structure $J$ on the spheres of the $C_n^r$ configuration. We can do so in the same manner as was done for the $C_n$ configuration in Lemma~\ref{l:jholocn}. The only difference is that we apply Lemma~\ref{l:localortho} to the small Darboux neighborhoods of the double points of the immersed sphere $S_r^d$, as well as to the small Darboux neighborhoods of the intersections between adjacent spheres in the plumbing.

As before, since $c_1^2(M,\omega) \leq -1$, by Corollary~\ref{cor:taubes} of the theorems of Taubes (Theorems~\ref{thm:Tmain} and \ref{thm:Tmin}), there must exist a $J_{\epsilon}$-holomorphic sphere $\hat{\Sigma}_{-1}^{\epsilon}$ in $(M,\omega)$, with $[\hat{\Sigma}_{-1}^{\epsilon}]^2=-1$ for a generic $\omega$-compatible almost-complex structure $J_{\epsilon}$. As in section~\ref{sec:step1}, the spheres of the $C_n^r$ configuration are $J$-holomorphic curves, and the sphere $\hat{\Sigma}_{-1}^{\epsilon}$ is a $J_{\epsilon}$ holomorphic curve. Therefore, we use Gromov Compactness (Theorem~\ref{thm:gromov}), and take a sequence of almost-complex structures $J_{\epsilon} \rightarrow J$ of which there exists a subsequence such that $\hat{\Sigma}_{-1}^{\epsilon}$ converges to a multicurve $\hat{u} = (\hat{u}^1, \ldots , \hat{u}^N)$. We can then apply Proposition~\ref{p:grcomp1}, and conclude that there exists at least one $i$, such that $\hat{u}^i$ is an embedded $J$-holomorphic sphere, which we will label by $\hat{\Sigma}_{-1}$.

Again, as before, we apply Lemma~\ref{l:lius}, to the $J$-holomorphic curves $S_r^d,$ $S_2,$ $S_3,$ $\ldots,$ $S_{n-1},$ $\hat{\Sigma}_{-1}$, and perturb these into symplectic surfaces $\hat{S}_r^d$, $S'_2$, $S'_3$, $\ldots$, $S'_{n-1}$, $\hat{\Sigma'}_{-1}$ which will intersect each other positively and transversally. The symplectic surface $\hat{S}_r^d$ has genus $g(\hat{S}_r^d) = d$, since it was obtained from the immersed sphere $S_r^d$ by smoothing out the double points, see \cite{LiUs}. However, we can replace $\hat{S}_r^d$ back with $S_r^d$, and consider the linear plumbing configuration of spheres $S_r^d, S'_2, S'_3, \ldots, S'_{n-1}$. We can still conclude that the sphere $\hat{\Sigma'}_{-1}$ (after a possible perturbation) intersects positively and transversally with that configuration, since $S_r^d$ differs from $\hat{S}_r^d$ only in small neighborhoods around its double points.  \end{proof}

\begin{prop}
\label{p:bnembedek}
Let $B_n \hookrightarrow (W,\omega)$, where $(W,\omega)$ is a symplectic $4$-manifold with $b_2^+(W) >1$, be an embedding of type $\E_k$, i.e. $I_1 = k$ and $I_j=0$ for $2 \leq j \leq n-1$, for $k \geq c_1^2(W,\omega) + 2$, then $(W,\omega)$ must have SW basic classes in addition to $\pm c_1(W,\omega)$.
\end{prop}

\begin{proof}
Assume $B_n \hookrightarrow (W,\omega)$ is an embedding of type $\E_k$. This implies that after symplectically rationally blowing up $(W,\omega)$, we obtain $(W',\omega')$ which contains a $C_n$ configuration of symplectic spheres, and a symplectic sphere $\Sigma_{-1}$ which intersects the sphere $S_1$ ($[S_1]^2 = -n-2$) $k$ times positively and transversally. 

We blow down the sphere $\Sigma_{-1}$, and obtain a manifold $(W^{(2)},\omega^{(2)})$, such that $c_1^2(W^{(2)},\omega^{(2)}) = c_1^2(W',\omega') + 1$. The sphere $S_1 \subset W$ descends to an immersed sphere $S_{-n-2+k^2}^{k-tuple}$ which has self-intersection $(-n-2+k^2)$ and a $k$-tuple intersection point. Since the sphere $S_1$ was in fact pseudo-holomorphic, and the blow-down map is holomorphic, the immersed sphere $S_{-n-2+k^2}^{k-tuple}$ is pseudo-holomorphic as well. Therefore, $S_{-n-2+k^2}^{k-tuple}$ can be perturbed to a pseudo-holomorphic sphere with only double point intersections (see \cite{McD1}), of which there will be $\frac{k(k-1)}{2}$ such double points. Consequently, the manifold $(W^{(2)},\omega^{(2)})$ will contain a linear configuration $C_{-n-2+k^2}^{k(k-1)/2}$ of spheres $S_{-n-2+k^2}^{k(k-1)/2},S_2,S_3, \ldots, S_{n-1}$, where $S_{-n-2+k^2}^{k(k-1)/2}$ is an immersed symplectic sphere with self-intersection $r=-n-2+k^2$ and $d=\frac{k(k-1)}{2}$ double points. 

Next, since $k \geq c_1^2(W,\omega) + 2$, we have that $c_1^2(W^{(2)},\omega^{(2)}) \leq -1$, therefore, we can apply Lemma~\ref{l:crn} and obtain an embedded symplectic sphere of self-intersection $(-1)$: $\Sigma_{-1}^{(2)} \subset W^{(2)}$. This sphere $\Sigma_{-1}^{(2)}$ must intersect the configuration $C_{-n-2+k^2}^{k(k-1)/2}$, since if it did not, we could blow up $(W^{(2)},\omega^{(2)})$, obtain $(W',\omega')$ again, rationally blow down and get $(W,\omega)$, which would contain the sphere $\Sigma_{-1}^{(2)}$, a contradiction since $c_1^2(W,\omega) \geq 1$. By Lemma~\ref{l:crn}, $\Sigma_{-1}^{(2)}$ must then intersect the spheres of the $C_{-n-2+k^2}^{k(k-1)/2}$ configuration positively and transversally.

By Lemma~\ref{l:int12}, if $\Sigma_{-1}^{(2)}$ intersects with the spheres $S_j$, $2 \leq j \leq n-1$ and $[S_j]^2 = -2$, then we must have $[\Sigma_{-1}^{(2)}] \cdot [S_j] = 1$. However, if this is the case, then we would be able to blow down repeatedly $(n-2)$ times and end up with a manifold that has a sphere of self-intersection $(-1)$ and $c_1^2 \geq 1$, which is a contradiction. Therefore, $\Sigma_{-1}^{(2)}$ must only intersect with the immersed sphere $S_{-n-2+k^2}^{k(k-1)/2}$.

Since $\Sigma_{-1}^{(2)}$ is a sphere of self-intersection $(-1)$, then $c_1(W^{(2)},\omega^{(2)}) + 2[\Sigma_{-1}^{(2)}]$ is a SW basic class of $W^{(2)}$, by Theorem~\ref{thm:genadjimm}. If we apply Theorem~\ref{thm:genadjimm} to $x = S_{-n-2+k^2}^{k(k-1)/2}$, we obtain:
\begin{equation}
|(c_1(W^{(2)},\omega^{(2)}) + 2[\Sigma_{-1}^{(2)}]) \cdot S_{-n-2+k^2}^{k(k-1)/2}| \leq n-k ,
\end{equation}

\noindent which implies that
\begin{equation}
[S_{-n-2+k^2}^{k(k-1)/2}] \cdot [\Sigma_{-1}^{(2)}] = j_2, \,\,\, 1 \leq j_2 \leq n-k .
\end{equation}

If $j_2 = n-k$, then we could blow up $(W^{(2)},\omega^{(2)})$ and obtain $(W',\omega')$, which would now contain $2$ spheres with self-intersection $(-1)$: $\Sigma_{-1}$ and $\Sigma_{-1}^{(2)}$, where:
\begin{eqnarray*}
\left[\Sigma_{-1}\right] \cdot \left[S_1\right] &=& k  \\
\left[\Sigma_{-1}^{(2)}\right] \cdot \left[S_1\right] &=& n-k .
\end{eqnarray*}

\noindent Since $([\Sigma_{-1}] + [\Sigma_{-1}^{(2)}]) \cdot [S_1] = n$, as in Lemma~\ref{l:I1n}, we can construct a relative class $\delta \in H_2(C_n , \partial C_n ;\Z)$ that is a union of the disks $(\Sigma_{-1} \cup \Sigma_{-1}^{(2)}) \cap N(C_n)$, such that the relative class $-n\delta$ can be supported in the interior by the following homology class:
\begin{equation}
\label{eq:ndelta}
-n\delta = (n-1)[S_1] + (n-2)[S_2] + \cdots + [S_{n-1}].
\end{equation}

\noindent As a result, as in the proof of Lemma~\ref{l:I1n}, the SW basic class $\pm(c_1(W',\omega')$ $ + 2[\Sigma_{-1}] + 2[\Sigma_{-1}^{(2)}])$ will extend to a SW basic class on $(W,\omega)$ after rationally blowing down, forcing $(W,\omega)$ to have basic classes in addition to $\pm c_1(W,\omega)$.

If $j_2 \neq n-k$, then we blow down the sphere $\Sigma_{-1}^{(2)}$ in $(W^{(2)},\omega^{(2)})$, and obtain the manifold $(W^{(3)},\omega^{(3)})$. The sphere $S_{-n-2+k^2}^{k(k-1)/2} \subset (W^{(2)},\omega^{(2)})$ descends to the sphere $S_{-n-2+k^2+j_2^2}^{(k(k-1) + j_2(j_2-1))/2} \subset (W^{(3)},\omega^{(3)})$, (after perturbing the $j_2$-tuple intersection, as done before). Next, since $k \geq c_1^2(W,\omega) + 2$, we have that $c_1^2(W^{(3)},\omega^{(3)}) \leq -1$, therefore, we can apply Lemma~\ref{l:crn} and obtain an embedded symplectic sphere of self-intersection $(-1)$: $\Sigma_{-1}^{(3)} \subset W^{(3)}$. Again, we have that $c_1(W^{(3)},\omega^{(3)})+2[\Sigma_{-1}^{(3)}]$ is a SW basic class of $(W^{(3)},\omega^{(3)})$, thus by Theorem~\ref{thm:genadjimm} we have that:
\begin{equation}
|(c_1(W^{(2)},\omega^{(2)}) + 2[\Sigma_{-1}^{(2)}]) \cdot S_{-n-2+k^2}^{k(k-1)/2}| \leq n-k ,
\end{equation}

\noindent which implies that 
\begin{equation}
[S_{-n-2+k^2+j_2^2}^{(k(k-1) + j_2(j_2-1))/2}] \cdot [\Sigma_{-1}^{(3)}] = j_3, \,\,\, 1 \leq j_2 \leq n-k-j_2 .
\end{equation}

If $j_3 = n-k-j_2$, then we could blow up $(W^{(3)},\omega^{(3)})$ twice and obtain $(W',\omega')$, which would now contain $3$ spheres with self-intersection $(-1)$: $\Sigma_{-1}, \Sigma_{-1}^{(2)}$ and $\Sigma_{-1}^{(3)}$, where:
\begin{eqnarray*}
\left[\Sigma_{-1}\right] \cdot \left[S_1\right] &=& k  \\
\left[\Sigma_{-1}^{(2)}\right] \cdot \left[S_1\right] &=& j_2  \\
\left[\Sigma_{-1}^{(3)}\right] \cdot \left[S_1\right] &=& n-k-j_2   .
\end{eqnarray*}

\noindent Since $([\Sigma_{-1}] + [\Sigma_{-1}^{(2)}]+[\Sigma_{-1}^{(3)}]) \cdot [S_1] = n$, again as in Lemma~\ref{l:I1n}, we can construct a relative class $\delta \in H_2(C_n , \partial C_n ;\Z)$ that is a union of the disks $(\Sigma_{-1} \cup \Sigma_{-1}^{(2)} \cup \Sigma_{-1}^{(3)}) \cap N(C_n)$, such that the relative class $-n\delta$ can be supported in the interior by the same class as before in (\ref{eq:ndelta}). As a result, just as in the proof of Lemma~\ref{l:I1n}, the SW basic class $\pm(c_1(W',\omega') + 2[\Sigma_{-1}] + 2[\Sigma_{-1}^{(2)}]+2[\Sigma_{-1}^{(3)}])$ will extend to a SW basic class on $(W,\omega)$ after rationally blowing down, again forcing $(W,\omega)$ to have basic classes in addition to $\pm c_1(W,\omega)$.

If $j_3 \neq n-k-j_2$, we can repeat the same procedure again, which will again force $(W,\omega)$ to have basic classes in addition to $\pm c_1(W,\omega)$. We can continue this process until it terminates for some ${\ell} \leq n-k$, where we will have a $j_{\ell}$ so that $j_2+j_3+j_4 + \cdots + j_{\ell}=n-k$. As a result, we will obtain the manifold $(W^{(\ell)},\omega^{(\ell)})$, which will have a sphere $S_{-n-2+k^2+j_2^2+\cdots+j_{\ell-1}^2}^{(k(k-1) + j_2(j_2 - 1) + \cdots + j_{\ell - 1}(j_{\ell - 1} - 1))/2}$ that interesects the sphere $\Sigma_{-1}^{(\ell)}$, $(n-k-j_2-j_3-\cdots -j_{\ell-1})$ times. We can then blow up $(W^{(\ell)},\omega^{(\ell)})$ $(\ell-1)$ times, and obtain the manifold $(W',\omega')$ which will have $\ell$ spheres of self-intersection $(-1)$, such that:
\begin{eqnarray*}
\left[\Sigma_{-1}\right] \cdot \left[S_1\right] &=& k  \\
\left[\Sigma_{-1}^{(2)}\right] \cdot \left[S_1\right] &=& j_2  \\
\left[\Sigma_{-1}^{(3)}\right] \cdot \left[S_1\right] &=& j_3  \\
\vdots &=& \vdots  \\
\left[\Sigma_{-1}^{(\ell-1)}\right] \cdot \left[S_1\right] &=& j_{\ell - 1}  \\
\left[\Sigma_{-1}^{(\ell)}\right] \cdot \left[S_1\right] &=& n - k - j_2 - j_3 - \cdots - j_{\ell - 1} = j_{\ell} .
\end{eqnarray*}

\noindent Again, in this case, we will have the SW basic class $\pm(c_1(W',\omega') + 2[\Sigma_{-1}] + 2[\Sigma_{-1}^{(2)}]+2[\Sigma_{-1}^{(3)}] + \cdots + 2[\Sigma_{-1}^{(\ell)}])$ which will extend to a SW basic class on $(W,\omega)$ after rationally blowing down, again forcing $(W,\omega)$ to have basic classes in addition to $\pm c_1(W,\omega)$.

Notice, that we will have to do the greatest number of blow downs if $j_2 = j_3 = \cdots = j_{\ell}=1$, in which case, $\ell = n-k$. Therefore, we require $k \geq c_1^2(W,\omega) + 2$, in order for all the manifolds $(W^{(i)},\omega^{(i)})$ with $1 \leq i \leq \ell$ to have $c_1^2(W^{(i)},\omega^{(i)}) \leq -1$, so that we can apply Lemma~\ref{l:crn} repeatedly.   \end{proof}

From Proposition~\ref{p:bnembedek}, we can see that if $B_n \hookrightarrow (X,\omega)$ is of type $\E_k$, $k \geq c_1(X,\omega)  + 2$, then $(X,\omega)$ must have SW basic classes in addition $\pm c_1(X,\omega)$, which is a contradiction.

\section{Symplectic embeddings of type $\E_2$}
\label{sec:e2}
In this section we will show how to explicitly construct symplectic $4$-manifolds $(X,\omega)$, such that the symplectic embeddings $B_n \hookrightarrow (X,\omega)$ are of type $\E_2$, for $n$ odd. In these constructions $(X,\omega)$ will have $b_2^+(X) > 1$, $n \geq c_1^2(X,\omega) + 2$ and $\mathcal{B}as_X\left\{\pm(c_1(X,\omega)\right\}$. It is not clear however, whether such a construction actually yields a surface of general type or just a symplectic $4$-manifold with said properties. Note, these constructions appear in \cite{Ak}, however, we reinterpret them here for our purposes. First, we introduce the Fintushel and Stern knot surgery construction for $4$-manifolds \cite{FS3,FS4}.

\begin{defn}
\label{def:knotsurg}
Let $T \subset X$ be a homologically non-trivial torus, with self-intersec\-tion $0$, in a $4$-manifold $X$ with $b_2^+(X) > 1$. Let $T \times D^2$ be a tubular neighborhood of $T$ in $X$. Also, let $K \subset S^3$ be a knot, and $N(K)$ be its tubular neighborhood. Then, 
\begin{equation}
X_K = (X \backslash (T \times D^2)) \cup (S^1 \times (S^3 \backslash N(K)))
\end{equation}

\noindent is defined to be the \textit{knot surgery manifold}.
\end{defn}

\noindent Note, the two pieces are attached in such a manner that the homology class $[* \times \partial D^2]$ is identified with $[* \times \lambda]$, where $\lambda$ is the longitude of the knot $K$. In other words, $X_K$ is obtained from $X$ by removing a neighborhood of the torus $T$ and replacing it with $(S^1 \times (S^3 \backslash N(K)))$. The manifold $X_K$ is homotopy equivalent to $X$ (assuming $X$ is simply-connected).

In \cite{FS3}, Fintushel and Stern proved that the Seiberg-Witten invariants of $X_K$ are determined by the Seiberg-Witten invariants of $X$ and the Alexander polynomial of the knot $K$, as long as $T$ has a cusp neighborhood. For the statement of this result, it is convenient to arrange all of the Seiberg-Witten basic classes into a Laurent polynomial as follows:

\begin{defn}
Let $\mathcal{B}as_X = \left\{\pm \beta_1, \ldots, \pm \beta_{m}\right\}$ and $t_{\beta_i} = exp({\beta_i})$ be variables satisfying $t_{\beta_i + \beta_j} = t_{\beta_i} t_{\beta_j}$, then 
\begin{equation}
\mathcal{SW}_X = b_0 + \sum_{i=1}^m b_i(t_{\beta_i} + (-1)^{(\chi(X) + \sigma(X))/4}t_{\beta_i}^{-1})
\end{equation}

\noindent where $b_0 = SW_X(0)$ and $b_i = SW_X(\beta_i)$.
\end{defn}

\begin{example}
Let $X=E(m)$ be the elliptic surface, and $t = exp(T)$, where $T$ is Poincaire dual of the fiber class, then:
\begin{equation}
\mathcal{SW}_{E(m)} = (t - t^{-1})^{m-2} \, .
\end{equation}
\end{example}

\begin{thm}
\label{thm:knotsurg}
Let $T \subset X$ be as above in Definition~\ref{def:knotsurg}. Assume that $T$ lies in a cusp neighborhood in $X$, then:
\begin{equation}
\mathcal{SW}_{X_K} = \mathcal{SW}_X \cdot \Delta_K(t)
\end{equation}

\noindent where $\Delta_K(t)$ is the Alexander polynomial of the knot $K$.
\end{thm}

\begin{rmk}
If $\Delta_K(t)$ is not monic then $X_K$ cannot admit a symplectic structure, since if $X_K$ is symplectic then we must have $SW_{X_K}(\pm c_1(X_K,\omega)) = \pm 1$. However, if the knot $K$ is fibered, then the \textit{knot surgery manifold} $X_K$ has a symplectic structure \cite{FS3}, since it can be constructed as a symplectic fiber sum \cite{Gompf2}.
\end{rmk}

We will exhibit symplectic $4$-manifolds which have symplectic embeddings $B_n \hookrightarrow X$ of type $\E_2$, by obtaining them from the elliptic surfaces $E(m)$ by knot surgery, blow-ups, and rational blow-down, (these constructions appeared in \cite{Ak}). We will utilize the following Lefschetz fibration of the elliptic surfaces $E(m)$:

\begin{lemma}
\label{l:ellfib}
\cite{Ak} There exists an elliptic Lefschetz fibration on the surface $E(m)$ with a section, a singular fiber F of type $I_{8m}$, $(2m-1)$ singular fibers of type $I_2$ and two additional fishtail fibers.
\end{lemma}

\begin{figure}[ht!]
\labellist
\small\hair 2pt
\pinlabel $\text{fishtail fiber}$ at 40 5
\pinlabel $0$ at 45 25
\pinlabel $I_2\,\text{fiber}$ at 115 5
\pinlabel $-2$ at 100 35
\pinlabel $-2$ at 127 35
\pinlabel $I_{l}\,\text{fiber}$ at 210 5
\pinlabel $-2$ at 192 30
\pinlabel $-2$ at 186 45
\pinlabel $-2$ at 200 60
\pinlabel $-2$ at 236 45
\pinlabel $-2$ at 223 60
\pinlabel $l$ at 230 15
\endlabellist
\centering
\includegraphics[height = 40mm, width = 120mm]{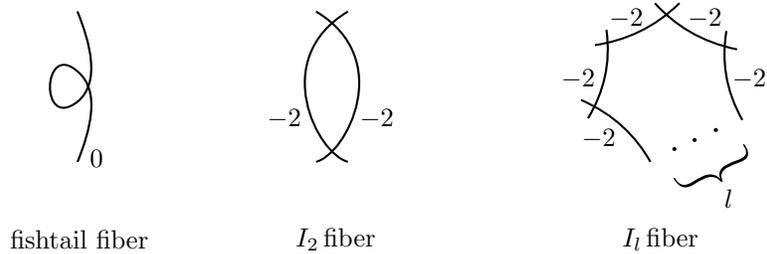}
\caption{Fibers in an elliptic fibration}
\label{f:ellfibers}
\end{figure} 

\noindent Recall, that a singular fiber of type $I_l$ is a plumbing of $l$ spheres of self-intersection $(-2)$ in a circle, and a fishtail fiber is an immersed sphere with one positive double point and self-intersection $0$ (for more on elliptic surfaces and their singular fibers, see \cite{HKK,KiMe}, also see Figure~\ref{f:ellfibers}).

In \cite{FS4}, Fintushel and Stern investigated the consequences of performing the \textit{knot surgery} construction in certain neighborhoods in an elliptic fibration:

\begin{defn}
\cite{FS4} A \textit{double node neighborhood} $D$ is a fibered neighborhood of an elliptic fibration which contains exactly two nodal fibers with the same monodromy.
\end{defn}

\noindent One can perform knot surgery along a regular fiber in such a double node neighborhood, $D$, for example, in a neighborhood of the $I_2$ fiber (see Figure~\ref{f:ellfibers}). The elliptic surface $E(m)$ will have a section $R$, which is a sphere with self-intersection $(-m)$. Fintushel and Stern observed that for a family of knots, the twists knots $T(r)$, if we perform knot surgery in the neighborhood of the $I_2$ singular fiber, then a disk in the section $R$ gets replaced by a Seifert surface of the knot $T(r)$. As a result, the manifold $E(m)_{T(r)}$, will have a ``pseudo-section" $R_s$, which we can think of as an immersed sphere with one double point (since $g(T(r)) = 1$), still having self-intersection $(-m)$ (see \cite{FS4,Ak}). Note, we will use this construction only for the knot $T(1)$, which is the trefoil knot, since we are interested in our $4$-manifolds retaining their symplectic structures. 

Next, we will describe the general construction of a family of such manifolds, similar to the examples above, (again, see \cite{Ak}).

\begin{prop}
\label{p:bne2}
There exists a family of symplectic $4$-manifolds $\X$, with each $(X,\omega) \in \X$ having $b_2^+(X) > 1$, $\mathcal{B}as_X = \left\{\pm c_1(X,\omega)\right\}$ and a symplectic embedding $B_n \hookrightarrow (X,\omega)$ of type $\E_2$, for $n$ odd. Moreover, for all $(X,\omega) \in \X$, the embeddings of $B_n \hookrightarrow (X,\omega)$ are such that $n < 3 + \frac{4}{3}c_1^2(X,\omega)$.
\end{prop}

\begin{proof}

\begin{figure}[ht!]
\labellist
\small\hair 2pt
\pinlabel $F_1$ at 35 15
\pinlabel $0$ at 20 5
\pinlabel $R_s$ at 65 60
\pinlabel $-m$ at 80 43
\pinlabel $F_2$ at 242 15
\pinlabel $0$ at 258 5
\pinlabel $I_{8m}$ at 345 30
\pinlabel $-2$ at 282 23
\pinlabel $-2$ at 275 43
\pinlabel $-2$ at 290 70
\pinlabel $-2$ at 332 48
\pinlabel $-2$ at 318 70
\pinlabel $8m$ at 325 0
\endlabellist
\centering
\includegraphics[height = 25mm, width = 120mm]{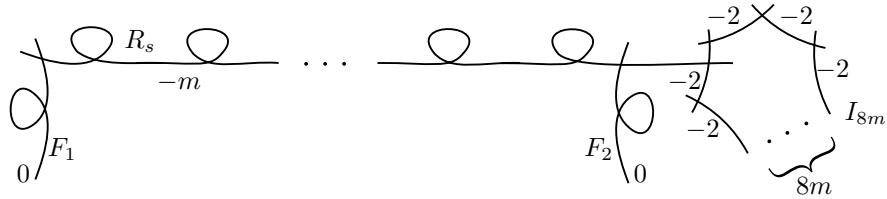}
\caption{Pseudo-section $R_s$ with fibers in $E(m)_{K_1, \ldots, K_s}$}
\label{f:ellex1b}
\end{figure} 

First, we take the elliptic surface $E(m)$, $m > 2$, which has a section $R$, a sphere of self-intersection $(-m)$, and perform knot surgery in the double node neighborhoods of $s$ of the $I_2$ fibers, obtaining the manifold $E(m)_{K_1, \ldots, K_s}$, where $1 \leq s \leq 2m-1$ and $K_i$ are copies of the trefoil knot. We now obtain a ``pseudo-section" $R_s$ (see \cite{FS4,Ak}) of  $E(m)_{K_1, \ldots, K_s}$, which is an immersed sphere with self-intersection $(-m)$ and $s$ double points (see Figure~\ref{f:ellex1b}).

\begin{figure}[ht!]
\labellist
\small\hair 2pt
\pinlabel $F_1$ at 60 15
\pinlabel $0$ at 45 5
\pinlabel $S_{-m-4s}$ at 30 60
\pinlabel $-m-4s$ at 30 45
\pinlabel $F_2$ at 228 15
\pinlabel $0$ at 244 5
\pinlabel $E_1$ at 72 40
\pinlabel $E_2$ at 115 40
\pinlabel $E_{s-1}$ at 170 40
\pinlabel $E_s$ at 220 40
\pinlabel $-1$ at 82 68
\pinlabel $-1$ at 125 68
\pinlabel $-1$ at 180 68
\pinlabel $-1$ at 227 68
\pinlabel $I_{8m}$ at 327 30
\pinlabel $-2$ at 264 23
\pinlabel $-2$ at 257 43
\pinlabel $-2$ at 272 70
\pinlabel $-2$ at 314 48
\pinlabel $-2$ at 300 70
\pinlabel $8m$ at 307 0
\endlabellist
\centering
\includegraphics[height = 25mm, width = 120mm]{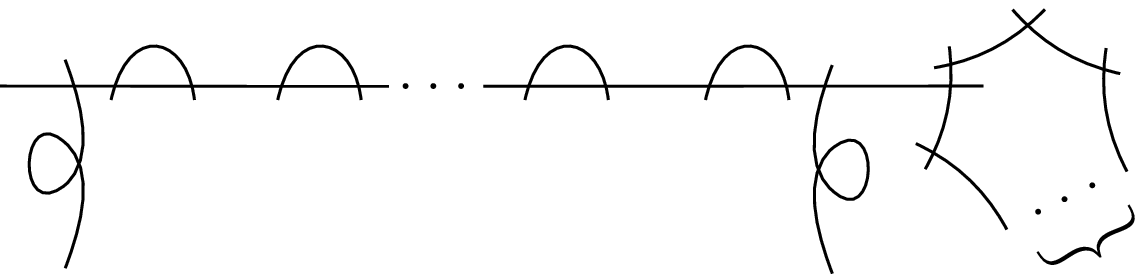}
\caption{ $E(m)_{K_1, \ldots, K_s}\# s\overline{\C P^2}$}
\label{f:ellex2b}
\end{figure} 

\begin{figure}[ht!]
\labellist
\small\hair 2pt
\pinlabel $E_{s+1}$ at 67 30
\pinlabel $-1$ at 38 25
\pinlabel $S_{-4}^{F_1}$ at 67 5
\pinlabel $-4$ at 45 5
\pinlabel $S_{-m-4s-2}$ at 25 60
\pinlabel $-m-4s-2$ at 25 45
\pinlabel $F_2$ at 226 15
\pinlabel $0$ at 242 5
\pinlabel $E_1$ at 75 40
\pinlabel $E_2$ at 118 40
\pinlabel $E_{s-1}$ at 170 40
\pinlabel $E_s$ at 215 40
\pinlabel $-1$ at 82 68
\pinlabel $-1$ at 125 68
\pinlabel $-1$ at 177 68
\pinlabel $-1$ at 227 68
\pinlabel $I_{8m}$ at 322 30
\pinlabel $-2$ at 259 23
\pinlabel $-2$ at 252 43
\pinlabel $-2$ at 267 70
\pinlabel $-2$ at 314 48
\pinlabel $-2$ at 300 70
\pinlabel $8m$ at 307 0
\endlabellist
\centering
\includegraphics[height = 25mm, width = 120mm]{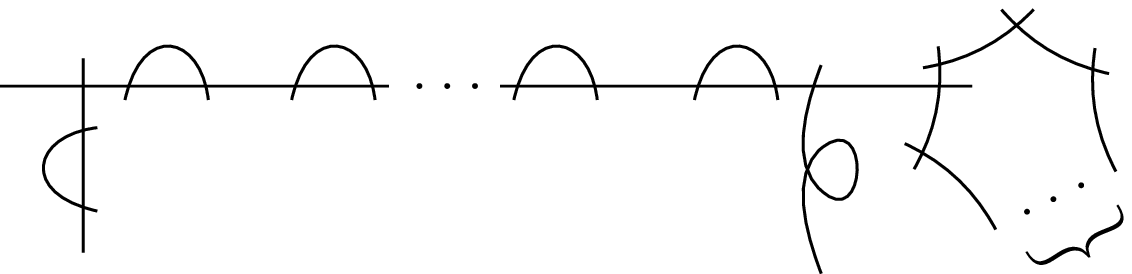}
\caption{ $E(m)_{K_1, \ldots, K_s}\# (s+1)\overline{\C P^2}$}
\label{f:ellex3b}
\end{figure} 

We can blow up $s$ times, so that $R_s$ becomes the embedded sphere $S_{-m-4s}$ (self-intersection $(-m-4s)$) in $E(m)_{K_1, \ldots, K_s} \# s \overline{\C P^2}$. Additionally, in $E(m)_{K_1, \ldots, K_s} \# s \overline{\C P^2}$ we will have $s$ exceptional spheres $E_1, \ldots, E_s$, with $[E_i]^2 = -1$, each of which intersects the sphere $S_{-m-4s}$ twice (see Figure~\ref{f:ellex2b}). In the fibration of $E(m)$, we also have two additional fishtail fibers, $F_1$ and $F_2$ (Lemma~\ref{l:ellfib}), which intersect the ``pseudo-section" $R_s$ once. Therefore, we can blow up $E(m)_{K_1, \ldots, K_s}$ $(s+1)$ times (at the double points of $R_s$ and the fishtail fiber $F_1$), and after smoothing out the transverse intersection, obtain a sphere $S_{-m-4s-2}$ in $E(m)_{K_1, \ldots, K_s} \# (s+1) \overline{\C P^2}$, such that $[S_{-m-4s-2}] = [S_{-4}^{F_1}] + [S_{-m-4s}]$ (see Figure~\ref{f:ellex3b}). Likewise, we can blow up $E(m)_{K_1, \ldots, K_s}$ $(s+2)$ times (at the double points of $R_s$ and the fishtail fibers $F_1$ and $F_2$), and after smoothing out the transverse intersections, obtain a sphere $S_{-m-4s-4}$ in $E(m)_{K_1, \ldots, K_s} \# (s+2) \overline{\C P^2}$, such that $[S_{-m-4s-4}] = [S_{-4}^{F_1}] + [S_{-m-4s}] + [S_{-4}^{F_2}]$. (The spheres $S_{-4}^{F_1}$ and $S_{-4}^{F_2}$ are spheres of self-intersection $(-4)$ obtained from blowing up fibers $F_1$ and $F_2$.)

In these three cases, we obtain configurations of $C_{m+4s-2}$, $C_{m+4s}$ and $C_{m+4s+2}$ in
\begin{eqnarray*}
E(m)_{K_1, \ldots, K_s} \hspace{-.1in} &\#& \hspace{-.1in} s \overline{\C P^2}  \\
E(m)_{K_1, \ldots, K_s} \hspace{-.1in} &\#& \hspace{-.1in} (s+1) \overline{\C P^2}  \\
E(m)_{K_1, \ldots, K_s} \hspace{-.1in} &\#& \hspace{-.1in} (s+2) \overline{\C P^2} \, , 
\end{eqnarray*}

\noindent respectively, by taking the spheres $S_{-m-4s}$, $S_{-m-4s-2}$ and $S_{-m-4s-4}$, also respectively, with the spheres of the $I_{8m}$ fiber. Note, this can be done as long we have enough spheres of self-intersection $(-2)$ in the $I_{8m}$ fiber to complete the $C_{m+4s-2}$, $C_{m+4s}$ and $C_{m+4s+2}$ configurations, so we must have $(8m-1) \geq (m+4s)$, $(8m-1) \geq (m+4s-2)$ or $(8m-1) \geq (m+4s-4)$, respectively. We can then rationally blow down these configurations and obtain manifolds $X_{(m+4s-2)}$, $X_{(m+4s)}$ and $X_{(m+4s+2)}$, such that:
\begin{eqnarray*}
B_{m+4s-2} &\hookrightarrow& X_{(m+4s-2)} \cong RBD(E(m)_{K_1, \ldots, K_s} \# s \overline{\C P^2})  \\
B_{m+4s} &\hookrightarrow& X_{(m+4s)} \cong RBD(E(m)_{K_1, \ldots, K_s} \# (s+1) \overline{\C P^2})  \\
B_{m+4s+2} &\hookrightarrow& X_{(m+4s+2)} \cong RBD(E(m)_{K_1, \ldots, K_s} \# (s+2) \overline{\C P^2}) \, . 
\end{eqnarray*}

In all of these cases, the embeddings of $B_n$ will be symplectic (since we used the trefoil knot in the knot surgery construction) and will be of type $\E_2$ (due to the exceptional spheres $E_i$). Again, if $m$ is odd, then only the top basic classes
\begin{equation*}
\pm(m+2s-2)T + E_1 + E_2 + \cdots + E_{r} 
\end{equation*}

\noindent of $E(m)_{K_1, \ldots, K_s} \# r\overline{\C P^2}$ can extend to the rational blow-down, where $r \in $ $\left\{s, s+1, s+2\right\}$, (this follows from results in \cite{Park}, also see \cite{Ak}). As a result, the manifolds $X_{(m+4s-2)}$, $X_{(m+4s)}$ and $X_{(m+4s+2)}$ will each only have one SW basic class, up to sign.

It is clear from these embeddings of the rational homology balls $B_n$, that if we want higher values of $n$, we are going to have to take higher values of $m$, thus, we need to increase the $b_2^+$. In these constructions, the number $n$ is mainly restricted by the number of spheres of self-intersection $(-2)$ in the $I_{8m}$ fiber which we use to construct the $C_n$ configuration of spheres. Consequently, even if we use all of the $(2m-1)$ of the $I_2$ for our knot surgery construction along with both of the fishtail fibers $F_1$ and $F_2$, and get a sphere $S_{-m-4s-4}$, we may not be able not construct a $C_n$ configuration of spheres with $n=m+4s+2$ if we have $(8m-1) < (n-1)$. For this reason, for each $m$, in order to get the highest possible value for $n$, we may have to use less than the $(2m-1)$ of the $I_2$ fibers in our knot surgery construction. Consequently, the highest $n$ which will work for these constructions is when $n=8m+1$, where we use all the $(8m-1)$ available spheres of the $I_{8m}$ fiber.

If $m = 4k+1$, for $k \geq 1$, then we have:
\begin{equation*}
B_{8m+1} \hookrightarrow X_{(8m+1)} \cong RBD(E(m)_{K_1, \ldots, K_{7k+2}} \# (7k+3) \overline{\C P^2}) \, ,
\end{equation*}

\noindent where $b_2^+(X_{(8m+1)}) = 2m-1$ and $c_1^2(X_{(8m+1)}) = 25k+5$.

If $m = 4k+3$, for $k \geq 1$, then we have:
\begin{equation*}
B_{8m+1} \hookrightarrow X_{(8m+1)} \cong RBD(E(m)_{K_1, \ldots, K_{7k+6}} \# (7k+6) \overline{\C P^2}) \, ,
\end{equation*}

\noindent where $b_2^+(X_{(8m+1)}) = 2m-1$ and $c_1^2(X_{(8m+1)}) = 25k+2$. 

As a result, we can see that as $(\chi_h,c_1^2) \rightarrow \oo$ then $n \rightarrow \oo$ as well. Moreover, in all these examples we have $n < 3 + \frac{4}{3}c_1^2$. If we take $m \geq 5$, we can refine this bound to $n < 3 + \frac{32}{25}c_1^2$. \end{proof}

It is important to note that it is not clear whether the examples in Proposition~\ref{p:bne2} yield surfaces of general type or just symplectic $4$-manifolds. Additionally, one could probably construct embeddings of type $\E_k$ for $k \geq 3$ having the same properties as those of type $\E_2$ in Proposition~\ref{p:bne2}. This might be done by defining the knot surgery construction in double node neighborhoods for fibered knots with higher genus than the trefoil knot.


\begin{thebibliography}{0}

\bibitem[Ak]{Ak}
A. Akhmedov, \emph{Construction of exotic smooth structures}, Topology and its Applications \textbf{154} (2007), no. 6, 1134-1140

\bibitem[At]{At}
M. F. Atiyah. \emph{Convexity and commuting Hamiltonians}, Bull. London Math. Soc., \textbf{14} (1982), no. 1, 1-15

\bibitem[BM]{BoMo}
M. Boucetta and P. Molino, \emph{G\'{e}om\'{e}trie globale des syst\`{e}mes hamiltoniens compl\`{e}tement
int\'{e}grables: fibrations lagrangiennes singuli\`{e}res et coordonn\'{e}es action-angle \`{a} singularit\'{e}s.}, C. R.
Acad. Sci. Paris S\'{e}r. I Math., \textbf{308} (1989), no. 13, 421-424

\bibitem[CH]{CaHa}
A. Casson and J. Harer, \emph{Some homology lens spaces which bound rational balls}, Pacific J. Math. \textbf{96} (1981), 23-36

\bibitem[De]{De}
T. Delzant, \emph{Hamiltoniens p\'{e}riodiques et images convexes de l'application moment}, Bull. Soc. Math. France, \textbf{116} (1988), no. 3, 315-339


\bibitem[FS1]{FS2}
R. Fintushel and R. Stern, \emph{Immersed spheres in $4$-manifolds and the immersed Thom conjecture}, Turkish J. Math. \textbf{19} (1995), no. 2, 145-157

\bibitem[FS2]{FSrbd}
R. Fintushel and R. Stern, \emph{Rational blowdowns of smooth 4-manifolds}, J. Diff. Geom. \textbf{46} (1997) 181-235

\bibitem[FS3]{FS3}
R. Fintushel and R. Stern, \emph{Knots, links, and 4-manifolds}, Invent. Math. \textbf{134} (1998), 363–400

\bibitem[FS4]{FS4}
R. Fintushel and R. Stern, \emph{Double node neighborhoods and families of simply connected $4$-manifolds with $b^+= 1$}, J. Amer. Math. Soc. \textbf{19} (2006), 171-180.


\bibitem[FM]{FrMo}
R. Friedman, J. W. Morgan. \emph{On the diffeomorphism types of certain algebraic surfaces}, I. J. Diff. Geom., \textbf{27} (1988), no. 2, 297-369

\bibitem[Go]{Gompf2}
R. Gompf, \emph{A new construction of symplectic manifolds}, Ann. of Math. \textbf{142} (1995), 527-595

\bibitem[GuSt]{GuSt}
V. Guillemin and S. Sternberg, \emph{Convexity properties of the moment mapping} Invent. Math., \textbf{67} (1982), no. 3, 491-513 

\bibitem[Gr]{Gr}
M. Gromov, \emph{Pseudoholomorphic curves in symplectic manifolds}, Invent. Math., \textbf{82} (1985), 307-347 

\bibitem[GoSt]{GS}
R. Gompf and A. Stipsicz, \emph{An introduction to 4-manifolds and Kirby calculus}, Graduate Studies in Mathematics 20, (American Mathematical Society, Providence, RI, (1999))

\bibitem[HKK]{HKK}
J. Harer, A. Kas and R. Kirby, \emph{Handlebody decompositions of complex surfaces}, Memoirs AMS \textbf{62} (1986), no. 350

\bibitem[Kh1]{Kh1}
T. Khodorovskiy, \emph{Smooth embeddings of rational homology balls}, arXiv:1212.4391 [math.GT]

\bibitem[Kh2]{Kh2}
T. Khodorovskiy, \emph{Symplectic rational blow-up}, arXiv:1303.2581 [math.SG]

\bibitem[KM]{KiMe}
R. Kirby and P. Melvin, \emph{The $E_8$ manifolds, singular fibers and handlebody decomposition}, Algebr. Geom. Topol. \textbf{3} (2003), 577-568

\bibitem[KSB]{KoSB}
J. Koll\'{a}r and N. I. Shepherd-Barron, \emph{Threefolds and deformations of surface singularities}, Invent. Math. \textbf{91} (1988), no. 2, 299-338

\bibitem[Ko]{Ko1}
D. Kotschick, \emph{The Seiberg-Witten invariants of symplectic four-manifolds (after C. H. Taubes)}, Seminaire Bourbaki, 1995/96, \textbf{241} (1997), no. 812, 195-220.

\bibitem[KM]{KM}
P. Kronheimer and T. Mrowka, \emph{The genus of embedded surfaces in the projective plane}, Math. Res. Letters \textbf{1} (1994), no. 6, 797–808

\bibitem[LM]{LeMa}
Y. Lekili and M. Maydanskiy , \emph{The symplectic topology of some rational homology balls}, to appear in Commentarii Mathematici Helvetici (accepted 2012)

\bibitem[Li]{Li}
T.-J. Li, \emph{Existence of symplectic surfaces}, in Geometry and topology of manifolds, Fields Inst. Commun., 47, AMS Providence, RI, 2005, 203–217

\bibitem[LU]{LiUs}
T.-J. Li, M. Usher, \emph{Symplectic forms and surfaces of negative square}, J. Symplectic Geom. \textbf{4} (2006), no. 1, 71-91.

\bibitem[Mc]{McD1}
D. McDuff, \emph{Immersed spheres in symplecic 4-manifolds}, Ann. Inst. Fourier, Grenoble \textbf{42}, (1992), no. 1-2, 369-392 

\bibitem[McPo]{McP}
D. McDuff \& L. Polterovich, \emph{Symplectic packings and algebraic geometry}, Invent. Math. \textbf{115} (1994), 405-429

\bibitem[MS1]{McS2}
D. McDuff, \& D.A. Salamon, \emph{J-Holomorphic Curves and Quantum Cohomology}, AMS, University Lecture Series, Vol. 6, Providence, Rhode Island, 1994

\bibitem[MS2]{McS3}
D. McDuff, \& D.A. Salamon, \emph{J-Holomorphic Curves and Symplectic Topology}, AMS Colloquium Publications, Vol. 52, 2004

\bibitem[MW]{MW}
M. Micallef and B. White, \emph{The structure of branch points in minimal surfaces and in pseudo- holomorphic curves}, Ann. Math. \textbf{139} (1994), 35-85

\bibitem[Mo]{Mo}
J. Morgan. \emph{The Seiberg-Witten equations and applications to the topology of smooth four-manifolds}, Vol. 44 of Math. Notes. Princeton University Press, Princeton, NJ, 1996

\bibitem[OzSz]{OzSz}
P. Ozsv\'{a}th and Z. Szab\'{o}, \emph{The symplectic Thom conjecture}, Ann. math.  \textbf{151} (2000), no. 2, 93-124

\bibitem[Pa1]{Park}
J. Park. \emph{Seiberg-Witten invariants of generalised rational blow-downs.} Bull. Austral. Math. Soc., \textbf{56} (1997), no. 3, 363-384

\bibitem[Sy1]{Sym1}
M. Symington, \emph{Symplectic rational blowdowns}, J. Diff. Geom. \textbf{50} (1998), 505-518

\bibitem[Sy2]{Sym2}
M. Symington, \emph{Generalized symplectic rational blowdowns}, Algebr. Geom. Topol. \textbf{1} (2001), 503-518

\bibitem[Sy3]{Sym3}
M. Symington, \emph{Four dimensions from two in symplectic topology}, Topology and Geometry of manifolds (Athens, GA, 2001)
Proc. Sympos. Pure Math. 71, Amer. Math. Soc., Providence, RI, 2003, 153-208

\bibitem[Ta1]{Ta3}
C. H. Taubes, \emph{The Seiberg-Witten invariants and symplectic forms}, Math. Res. Letters \textbf{1} (1994), 809-822

\bibitem[Ta2]{Ta1}
C. H. Taubes, \emph{Counting pseudo-holomorphic cirves in dimension 4}, J. Diff. Geom. \textbf{44} (1996), 818-893

\bibitem[Ta3]{Ta4}
C. H. Taubes, \emph{$SW \Rightarrow Gr$: from the Seiberg-Witten equations to pseudoholomorphic curves}, J. Amer. Math. Soc. \textbf{9} (1996), 845-918

\bibitem[Ta4]{Ta2}
C. H. Taubes, \emph{Sieberg-Witten and Gromov invariants}, Geometry and Physics (Aarhus, 1995), Lecture notes in pure and applied math., 184, Dekker, New York 1997, 591-601

\bibitem[Wa]{Wa}
J. Wahl, \emph{Miyaoka-Yau inequality for normal surfaces and local analogues}, Contemporary Mathematics \textbf{162} (1994), 381-402


\end{thebibliography}
\end{document}